\documentclass[a4paper,11pt]{amsart}

\usepackage{amssymb} 
\usepackage{xcolor} 
\usepackage{fancybox}
\usepackage{dashbox}
\newboolean{showcomments}
\setboolean{showcomments}{true}
\ifthenelse{\boolean{showcomments}}
{ \newcommand{\mynote}[3]{
  \fbox{\bfseries\sffamily\scriptsize#1}
  {\small$\blacktriangleright$\textsf{\emph{\color{#3}{#2}}}$\blacktriangleleft$}}}
{ \newcommand{\mynote}[3]{}}
\definecolor{asparagus}{rgb}{0.53, 0.66, 0.42}

\usepackage{mathrsfs}

  \usepackage{float}
  \usepackage{pgffor}
  \usepackage{tikz}
  \usetikzlibrary{calc,3d,arrows,positioning,shapes.misc, plotmarks, shapes}
  \usepackage{url}
  \usepackage[english]{babel}
  \usepackage{amsthm}
  \usepackage{amsmath}
  \usepackage{amssymb}
  \usepackage{mathtools}
  \usepackage[latin1]{inputenc}
  \usepackage{setspace}
  \usepackage{enumerate}
  \usepackage{stmaryrd}
  \usepackage{algpseudocode}
  \usepackage{algorithm}
  \usepackage{multirow}
  \usepackage{adjustbox}

\definecolor{phase1}{HTML}{377EB8}
\definecolor{phase2}{HTML}{FF7F00}
\definecolor{phase3}{HTML}{4DAF4A}
\usepackage{accents}
\newcommand{\ubar}[1]{\underaccent{\bar}{#1}}

  \usepackage[colorlinks=true,linkcolor=black,citecolor=black,filecolor=black,urlcolor=black,menucolor=black]{hyperref}

  \usepackage{array}
  \usepackage{graphicx}
  \usepackage{xcolor}
  \usepackage{wrapfig}
  \usepackage{subfigure}
  \usepackage[babel]{csquotes}
  \usepackage{geometry}
  \usepackage{bbm}
  \usepackage{stmaryrd}
  \usepackage[all]{xy}
  \usepackage{mathrsfs}
  \usepackage{amsaddr}

  \theoremstyle{plain}
  \newtheorem{theorem}{Theorem}
  \newtheorem*{theorem*}{Theorem}
  \newtheorem{proposition}{Proposition}
  \newtheorem{corollary}{Corollary}
  
  \newtheorem{lemma}{Lemma}

  \theoremstyle{definition}
  \newtheorem{definition}{Definition}

  \newtheorem{rem}{Remark}

	\allowdisplaybreaks[3]

  \def \N{\mathbb{N}}
  \def \R{\mathbb{R}}
  \def \T{\mathbb{T}}
  \def \Z{\mathbb{Z}}
  
	\def \S{\mathbb{S}}

  \newcommand {\sign} {\mathrm{sign}}

 \newcommand {\Int} {\mathop \textup{Int}}
 \newcommand {\dist} {\textup{dist}}
 
 \newcommand {\tr} {\textup{tr}}
 \newcommand {\id} {\textup{id}}
 
 \newcommand {\sdist} {\textup{sdist}}

  \newcommand {\eps} {\varepsilon}
 
 \newcommand{\dom}{\Omega}
 \newcommand{\schvar}{u}

\DeclareMathOperator*{\argmin}{\mathop \textup{arg\,min}}

\DeclareMathAlphabet{\mathbbold}{U}{bbold}{m}{n}
 
\makeatletter
 \def\namedlabel#1#2{\begingroup
    #2%
    \def\@currentlabel{#2}%
    \phantomsection\label{#1}\endgroup
}
\makeatother

\title[Obstacle MBO]{Obstacle Mean Curvature Flow: Efficient Approximation and Convergence Analysis}
\author{Fabius Kr{\"a}mer, Tim Laux}

\address{Institut f{\"ur} Mathematik, Universit{\"a}t Heidelberg, Im Neuenheimer Feld 205, 69120 Heidelberg, Germany (\texttt{\{f.kraemer,tim.laux\}@math.uni-heidelberg.de})}
\date{\today}

\begin{document}
    
  \begin{abstract}
    
We introduce a simple and efficient numerical method to compute mean curvature flow with obstacles.
The method augments the Merrimam-Bence-Osher scheme with a pointwise update that enforces the constraint and therefore retains the computational complexity of the original scheme. 
Remarkably, this naive scheme inherits both crucial structural properties of obstacle mean curvature flow: a geometric comparison principle and a minimizing movements interpretation. 
The latter immediately implies the unconditional stability of the scheme.
Based on the comparison principle we prove the convergence of the scheme to the viscosity solution of obstacle mean curvature flow.
Moreover, using the minimizing movements interpretation, we show convergence of a spatially discrete model.
Finally, we present numerical experiments for a physical model that inspired this work.
	\medskip
    

  \medskip

\noindent \textbf{Mathematical Subject Classification (MSC2020)}:
53E10; 
74H15 (Primary); 
35D40; 
35R35; 
49Q05 (Secondary)
  \end{abstract}
\maketitle


\tableofcontents

\section{Introduction}\label{sec:intro}
Motivated by recent physical studies~\cite{martin2023domain} we introduce a novel approximation to mean curvature flow with obstacles. To that aim we adapt the famous Merriman-Bence-Osher (MBO) scheme~\cite{zbMATH00622264} that is well known to discretize mean curvature flow. The new scheme, although computationally cheap, inherits the properties of unconditional stability, monotonicity, has a minimizing movement interpretation and satisfies the obstacle constraints exactly. We prove convergence of our method to obstacle mean curvature flow in the viscosity sense. Furthermore, we demonstrate how to discretize the scheme spatially and establish the convergence of the spatially discretized schemes. 

The obstacle mean curvature flow describes the geometric evolution of a hypersurface whose normal velocity $V$ is given by its scalar mean curvature $H$, except at points where this motion would cause the surface to penetrate the obstacle. In such regions, the evolution is constrained so that the hypersurface may remain in contact with, but does not enter, the obstacle. Formally, this equations reads
\begin{equation}\label{def:mcf}
  V  = \begin{cases}
    \max\{0, -H\} & \text{on } \partial \Phi,\\
    -H & \text{in } (\Phi \cup \Psi)^c,\\
    \min\{0, -H\} & \text{on } \partial \Psi,
  \end{cases}
\end{equation}
where $\Phi$ is the inner and $\Psi$ the outer obstacle (see Fig.~\ref{fig:init_cond}).

Obstacle mean curvature flow is a model flow related to phenomena in biology, for example cell motility, see \cite{elliott2012modelling}. Another very common flow in biology is positive mean curvature flow $V =  \max\{0, -H\}$ which only evolves the surface in positive curved parts and is for example a model for wound healing and tissue repairing~\cite{zbMATH05605510,zbMATH06112667}. This is a special case of obstacle mean curvature flow. Indeed, one can reconstruct any given positive mean curvature flow by placing obstacles precisely at the negatively curved regions \cite{zbMATH06099202}.

From a mathematical perspective there are several works regarding the existence, uniqueness and regularity of mean curvature flow with obstacles, see for example~\cite{zbMATH06524747, zbMATH06099202, zbMATH06841741}.

Our main motivation for introducing an efficient scheme for the approximation of obstacle mean curvature flow is the domain convexification model for invasion processes of \cite{martin2023domain}. Their model describes for example spreading of species in ecological models, phase transition in magnetization or de-adhesion transition, which have in common that they are all to some extent curvature driven motions. In their work~\cite{martin2023domain} they are mainly interested in finding the density ratio of the initial conditions such that the flow floods the whole domain. 
Their initial condition is a union of disks that are randomly seeded in the plane.
To roughly approximate the positive curvature flow from the union of disks, they use that the flow converges to the (mean-) convex hull. Hence, they only compute the convex hull of clusters in every iteration of the scheme but not the dynamics that lead to it. We are interested in filling this gap by providing a simple scheme to efficiently compute the precise dynamics.

For given obstacles $\Phi, \Psi \subset \dom$ (where $\dom := \R^d$ or $\dom:= \mathbb{T}^d$) and an initial phase described by $\schvar^0: \dom \rightarrow \{-1,1\}$ we investigate the following iterative scheme.

\begin{itemize}
  \item[\textbf{Step 1}] Diffusion: Compute for fixed diffusion time $h$
  \begin{equation*}
    \schvar_h^{\ell +1/3} = e^{h \Delta}\schvar^{\ell}
  \end{equation*}
  where we denote by $e^{h \Delta}$ the heat semi-group (with the sign convention such that the Laplacian $\Delta$ has non-positive eigenvalues). 
  \item[\textbf{Step 2}] Threshold: Reconstruct a $\pm 1$-valued function by setting
  \begin{equation*}
    \schvar_h^{\ell + 2/3} =\begin{cases}
    1 & \text{if } \schvar_h^{\ell +1/3 } > 0,\\
    -1 & \text{otherwise}. 
    \end{cases}
  \end{equation*}
  \item[\textbf{Step 3}] Update on obstacles: Ensure that the obstacle constraint is satisfied by setting
  \begin{equation*}
    \schvar_h^{\ell +1}(x)  = \begin{cases}
      1 & \text{if } x \in \Phi,\\
      \schvar_h^{\ell + 2/3}(x) & \text{if } x \in \dom \setminus(\Phi \cup \Psi),\\
      -1 & \text{if } x \in \Psi.
    \end{cases} 
      \end{equation*}
\end{itemize}
Note that Step 3 of the scheme can also be written like
\begin{equation*}
\schvar_h^{\ell +1 } = \chi_\Phi \vee \schvar^{\ell + 2/3} \wedge (-\chi_\Psi) ,
\end{equation*}
where we denote
\begin{equation*}
  \chi_\Phi (x) := \begin{cases}
    1 & \text{if } x \in \Phi,\\
    -1 & \text{if } x \notin \Phi,
  \end{cases}
  \quad \quad  
  \chi_\Psi (x) := \begin{cases}
    1 & \text{if } x \in \Psi,\\
    -1 & \text{if } x \notin \Psi,
  \end{cases}
\end{equation*}
as well as $a \wedge b = \min\{a, b\}$ and $a \vee b = \max\{a, b\}$. 
If the obstacles are non-trivial, i.e. $\Phi \cup \Psi \neq \emptyset$, one can use the fact that $\schvar^{\ell + 1/3} \in (-1,1)$ and denote $\mathbbold{1}_S(x) := \begin{cases}
  1 & \text{if } x \in S,\\
  0 & \text{otherwise}
\end{cases}$  for any $S \subset \dom$ and $x \in \dom$ to shortly write the above scheme as 
\begin{equation}\label{def:scheme}
  \schvar^{\ell +1} = \sign(e^{h \Delta} \schvar^{\ell} + \mathbbold{1}_\Phi - \mathbbold{1}_\Psi).
\end{equation}
One recognizes that the scheme without Step 3 is exactly the thresholding scheme of Merriman, Bence and Osher~\cite{zbMATH00622264}. The convergence of the thresholding scheme to mean curvature flow (as $h \downarrow 0$) is a well studied subject~\cite{zbMATH00550306, zbMATH00763426, zbMATH01336919,zbMATH06654682,zbMATH07161205, zbMATH07373006,
arXiv:2506.05946, zbMATH07896008, arXiv:2508.09064}. Hence, heuristically, the above scheme~\eqref{def:scheme} moves with a discrete mean curvature motion away from the obstacles while the new Step 3 guarantees the constraints imposed by the obstacles. So, on a heuristic level, it should be of no surprise that the above scheme indeed converges to obstacle mean curvature flow. This will be proven later on in the viscosity solution setting, see Theorem~\ref{the:main}. One can view the scheme also from a different viewpoint to underline that our treatment of the obstacle is meaningful. In Lemma~\ref{lem:minmov} we show that adapting the minimizing movement interpretation of \cite{zbMATH06430104} to our problem yields that the scheme~\eqref{def:scheme} is equivalent to
\begin{equation*}
  \schvar^{\ell +1} \in \argmin_{\schvar:\dom \rightarrow \{-1,1\}} E_h(\schvar) + \frac{1}{2h}d_h^2(\schvar, \schvar^\ell) \quad \quad \text{ s.t. } \quad \quad \schvar \equiv 1 \text{ on } \Phi \quad \text{ and } \quad \schvar \equiv - 1 \text { on } \Psi,
\end{equation*}
with the energy and distance given by 
\begin{align}\label{def:energy}
  E_h(\schvar) &= \frac{1}{\sqrt{h}} \int_{\dom} (1- \schvar) e^{h \Delta} (1 + \schvar) \, dx, \nonumber\\
  \frac{1}{2h}d_h^2(\schvar, \tilde{\schvar}) &= \frac{1}{\sqrt{h}} \int_{\dom} (\schvar- \tilde{\schvar}) e^{h \Delta} (\schvar - \tilde{\schvar}) \, dx.
\end{align}
Remember, mean curvature flow is the (formal) gradient flow of  the perimeter and it is well known that $E_h$ converges to the perimeter functional. So, the minimizing movement interpretation reflects on the discrete level the gradient flow structure of obstacle mean curvature flow by simply adding the obstacles as constraints in each step. 
Using the minimizing movement interpretation to introduce constraints in the discrete flow has shown already to be sucessful in the case of volume constraints~\cite{zbMATH06841827, arXiv:2412.17694}.

Furthermore, the minimizing movement interpretation~\eqref{eq:minmov} immediately implies unconditional stability of the scheme. Another important property of the scheme is that because Step 1, Step 2 and Step 3 above are monotone, the whole scheme is monotone. This means, if $\schvar_h^\ell$ are the iterates with initial condition $\schvar^0$, and $\tilde{\schvar}_h^\ell$ the the iterates with initial condition $\tilde{\schvar}^0$, then $\schvar^0 \leq \tilde{\schvar}^0$ implies $\schvar^\ell_h \leq \tilde{\schvar}^\ell_h$ for any $\ell \in \N$. We will use both of these structural properties to understand the asymptotic behavior of the scheme.

The rest of the paper is structured as follows. In Section~\ref{sec:convergence_time}, we prove convergence of the scheme~\eqref{def:scheme} to the viscosity solution of obstacle mean curvature flow. There, the main results are given by Theorem~\ref{the:main} and Corollary~\ref{cor:conv_front}. They state, respectively, that the limit of~\eqref{def:scheme} is correct away from the possible fattened interface and the convergence of the fronts in case of no fattening. Section~\ref{sec:convergence_space} shows for a broad class of space discretizations the convergence to the space-continuous scheme. Finally, we apply in Section~\ref{sec:numerics} our space- and time-discretized scheme to simulate the invasion processes from~\cite{martin2023domain}.

\section{Convergence to Obstacle Mean Curvature Flow}\label{sec:convergence_time}
In the following we write $\Omega_0 := \{\schvar^0 = 1\}$, $\Gamma_0 := \partial \Omega_0$. Furthermore, we will assume that $\Omega_0, \Phi \subset \Omega_0, \Psi\subset \Omega_0^c$ are open, bounded and with $W^{2,\infty}$ boundary. 
Our aim is to prove the convergence of the obstacle MBO scheme~\eqref{def:scheme} to the viscosity solution of obstacle MCF. We follow the line of~\cite{zbMATH00205134, zbMATH06841741} for general properties of viscosity solutions of Obstacle Mean Curvature Flow. With these properties at hand, the remainder of the convergence proof is almost identical to the very elegant proof of~\cite{zbMATH00763426} for the MBO scheme without obstacles. 

Let us define $F: (\R^d\setminus \{0\}) \times \R^{d \times d}_{\text{sym}} \rightarrow \R$ by $F(p, X) := - \tr (X - \frac{p \otimes p}{|p|^2}X)$ which is geometric, i.e.,
\begin{equation}\label{def:geometric}
  F(\lambda p, \lambda X + \mu(p \otimes p)) = \lambda F(p,x) \quad \quad \text{ for any } \lambda > 0, \mu \in \R.
\end{equation}
Note that if the level sets of a function $u:\dom \rightarrow \R$ are smooth and evolve by mean curvature flow, it holds 
\begin{equation} \label{eq:level_set_equation}
  \partial_t u + F(Du, D^2u) = 0 \quad \quad \text{ in } \dom \times (0,\infty).
\end{equation} 
In the following we will denote for any function $f$ the upper and lower semi-continuous envelope by 
\begin{equation*}
  f^*(x) := \limsup_{y \rightarrow x} f(y) \quad \quad \text{ and } \quad \quad f_*(x) := \liminf_{y \rightarrow x} f(y),
\end{equation*}
respectively. 
For any symmetric matrix $X$ there is no possibility to define $F(0,X)$ such that $F$ is continuous at this point. But independently of the choice at $F(0,X)$, the lower semi-continuous envelope $F_*$ is given by
\begin{equation}\label{eq:lowerF}
  F_*(p,X) = \liminf_{\tilde{p} \rightarrow p} F(\tilde{p}, X) = \begin{cases}
    F(p,X) &\text{if } p \neq 0,\\
    - \tr(X) + \lambda_{\text{min}}(X) & \text{if } p = 0.
  \end{cases}
\end{equation} 
The upper semi-continuous envelope is similarly
\begin{equation*}
    F^*(p,X) = \limsup_{\tilde{p} \rightarrow p} F(\tilde{p}, X) = \begin{cases}
    F(p,X) &\text{if } p \neq 0,\\
    - \tr(X) + \lambda_{\text{max}}(X) & \text{if } p = 0.
    \end{cases}
\end{equation*}
Now we can define viscosity solutions for discontinuous Hamiltonians with obstacles $\phi, \psi$, see \cite{zbMATH06841741}. 
\begin{definition}\label{def:sub_super_sol}
  We call a locally bounded function $u: \dom \rightarrow \R$ a viscosity subsolution of~\eqref{eq:level_set_equation} with obstacles $\phi, \psi: \dom \rightarrow \R$ if for every $\eta \in C^2(\dom)$ with local maximum at $(x_0,t_0) \in \dom \times [0,T)$ of $u^* - \eta$, it holds
  \begin{equation}\label{eq:subsol_eq}
    \max\{\min\{\partial_t \eta + F_*(D \eta, D^2 \eta), u^* - \phi^*\}, u^* - \psi^*\} \leq 0.
  \end{equation} 
  
  Similarly, a supersolution is defined by satisfying
  for every $\eta \in C^2(\dom)$, if $u_* - \eta$  has a local minimum at $(x_0,t_0) \in \dom \times [0,T)$, then
  \begin{equation}\label{eq:supersol_eq}
    \max\{\min\{\partial_t \eta + F^*(D \eta, D^2 \eta), u_* - \phi_*\}, u_* - \psi_*\} \geq 0.
  \end{equation} 
\end{definition}
\begin{definition}\label{def:visc_sol}
  A locally bounded function $u:\dom \rightarrow \R$ is called a viscosity solution of~\eqref{eq:level_set_equation} with obstacles $\phi, \psi$ if it is a viscosity subsolution and a viscosity supersolution with respect to $\phi, \psi$. 
\end{definition}

\begin{rem}\label{rem:hamiltonian}
  Note that the above Definition~\ref{def:visc_sol} is just the classical definition for discontinuous Hamiltonians (see~\cite{zbMATH06841741}), where our Hamiltonian is given by
  \begin{equation*}
    H(u, s, p, A):= \max\{\min\{s + F(p, A), u - \phi\}, u- \psi\}.
  \end{equation*}
\end{rem}
\begin{rem}\label{rem:solution}
  Let us abbreviate $\mathcal{F} = \partial_t u + F(D\phi, D^2 \phi)$ for this remark. Although the expression $\max\{\min\{\mathcal{F}, u-\phi\}, u-\psi\}$ in~\eqref{eq:subsol_eq}, \eqref{eq:supersol_eq} looks asymmetric in $\phi, \psi$ at first glance, this is not the case. One checks, using $\phi \leq \psi$, the facts
  \begin{equation*}
    \max\{\min\{\mathcal{F}, u-\phi\}, u-\psi\} \leq 0 \Leftrightarrow (\mathcal{F}\leq 0 \vee u \leq \phi) \wedge u \leq \psi,
  \end{equation*}
  and 
  \begin{equation*}
    \max\{\min\{\mathcal{F}, u-\phi\}, u-\psi\} \geq 0 \Leftrightarrow (\mathcal{F}\geq 0 \vee u \geq \psi) \wedge u \geq \phi,
  \end{equation*}
  which show the symmetry in $\phi, \psi$. Here, one can also see nicely that every function $u$ with $u \leq \phi$ is a viscosity subsolution although the obstacle constraint $u \geq \phi$ might be violated. This flexibility of subsolutions will be crucial later for constructing test functions for the comparison principle.
\end{rem}
It was already proven in~\cite{zbMATH06841741} that there exists a unique viscosity solution to~\eqref{eq:level_set_equation} if one assumes enough regularity for the initial condition and the obstacles. Denote by $W^{2,1}_\infty(\Omega \times [0, \infty))$ the space of functions $f \in L^\infty(\Omega \times[0, \infty))$ with weak derivatives $\partial_t f, \nabla f, \nabla^2 f \in L^\infty[\Omega \times(0, \infty))$. 
\begin{theorem}[Theorem 4.2 in \cite{zbMATH06841741}]\label{the:existence}
  Let $u_0 \in W^{2, \infty}(\Omega)$ and $\phi, \psi \in W^{2,1}_\infty(\Omega \times [0, \infty))$. If $\phi \leq u_0 \leq \psi$ then there exists a unique viscosity solution $u \in C(\dom \times [0, \infty))$ of~\eqref{eq:level_set_equation} with 
  $u(x,0) = u_0(x)$ in $\dom$. 
\end{theorem}

For any $\Gamma_0 = \partial \Omega_0,\,\Phi,\,\Psi$ as given in the Introduction~\ref{sec:intro} (see Fig.~\ref{fig:init_cond}) we can construct in Theorem \ref{the:existence} an admissible initial condition $u_0 \in W^{2,\infty}(\dom)$ and obstacles $\phi, \psi \in W^{2,1}_\infty(\Omega \times [0, \infty))$ such that 
\begin{align}
  \Gamma_0 &= \{x\in \dom: u_0(x) = 0\}, \label{eq:level0u}\\
  \Phi &= \{x\in \dom: \phi(x) > 0\}, \label{eq:level0phi}\\
  \Psi &= \{x\in \dom: \psi(x) < 0\} \label{eq:level0psi}.
\end{align}
An example for such functions can be given by using $u_0 (x) = \sigma(\sdist(x, \Omega_0)/\eps)$ for $\eps$ small enough (choose $\eps$ such that the nearest point projection of $\{x \in \dom: \dist(x, \Omega_0) \leq \eps\}$ onto $\Gamma_0$, $\{x \in \dom: \dist(x, \Phi) \leq \eps\}$ onto $\Phi$ and $\{x \in \dom: \dist(x, \Psi) \leq \eps\}$ onto $\Psi$, are unique) and 
\begin{equation*}
  \sigma \in C^\infty(\R), s \mapsto \begin{cases}
    1 & \text{if } s \geq 1,\\
    1 - \frac{2}{1 + e^{\frac{s}{(1-s^2)(1+s)}}} & \text{if } -1< s< 1,\\
    -1 & \text{if } s \leq -1,
  \end{cases} \quad \quad \sdist(x, S) = \begin{cases}
    \dist(x ,S) & \text{if } x \in S,\\
    -\dist(x, S) & \text{if } x \notin S,
  \end{cases}
\end{equation*}
together with $\phi(x) = \sigma(\sdist(x, \Phi)/\eps)$ and $\psi(x) = -\sigma(\sdist(x, \Psi)/\eps)$. The equations~\eqref{eq:level0u}, \eqref{eq:level0phi}, and~\eqref{eq:level0psi} are immediately satisfied as $\sigma(s) = 0$ if and only if $s = 0$. As $\Phi$ lies inside $\Gamma_0$ it holds $\sdist(\cdot, \Gamma_0) \geq \sdist(\cdot, \Phi)$ \big(for $x$ outside $\Gamma_0$ it holds $-\dist(x, \Gamma_0) \leq -\dist(\cdot, \Phi)$; for $x$ inside $\Phi$ it holds $\dist(x, \Gamma_0) \geq \dist(\cdot, \Phi)$; for $x$ inside $\Gamma_0$ but outside $\Phi$ it holds $\dist(x, \Gamma_0) \geq -\dist(\cdot, \Phi)$\big). Similarly, it holds $\sdist(\cdot, \Gamma_0) \leq \sdist(\cdot, \Psi)$ such that \[\phi \leq u_0 \leq \psi\] is satisfied. Furthermore, the three functions $u_0, \phi, \psi$ are in $W^{2, \infty}(\dom)$ as $\eps$ was chosen small enough to make the distance function as regular as $\Gamma_0, \Phi, \Psi$, respectively.

In the following we will denote by $u$ the unique viscosity solution to the just constructed obstacles $\phi, \psi$ and the initial condition $u_0$, which exists by Theorem~\ref{the:existence}. The idea behind the construction of $u$ is that $\Gamma_t = \{x \in \dom: u(x,t) = 0\}$ is independent of the choice of $\phi, \psi$ and  $u_0$, but solely depends on $\Phi,\Psi$ and $\Gamma_0$. This motivates the next Proposition~\ref{the:max_min_sub_super} stating the maximal and minimal sub- and supersolution, respectively, are simply characterized by the sign of $u$.
\begin{figure}
  \centering
\tikzset{every picture/.style={line width=0.75pt}} 

\begin{tikzpicture}[x=0.75pt,y=0.75pt,yscale=-1,xscale=1]

\draw  [color={rgb, 255:red, 74; green, 144; blue, 226 }  ,draw opacity=1 ] (118.54,141.93) .. controls (118.36,130.2) and (127.72,120.55) .. (139.45,120.37) .. controls (151.17,120.19) and (160.82,129.55) .. (161,141.28) .. controls (161.18,153) and (151.83,162.65) .. (140.1,162.83) .. controls (128.37,163.01) and (118.72,153.65) .. (118.54,141.93) -- cycle ;
\draw  [color={rgb, 255:red, 208; green, 2; blue, 27 }  ,draw opacity=1 ] (229.05,145.41) .. controls (228.83,130.75) and (240.53,118.68) .. (255.19,118.46) .. controls (269.85,118.23) and (281.92,129.94) .. (282.15,144.6) .. controls (282.37,159.26) and (270.67,171.33) .. (256.01,171.55) .. controls (241.35,171.78) and (229.28,160.08) .. (229.05,145.41) -- cycle ;
\draw   (189.13,102.27) .. controls (207.13,102.27) and (301.4,76.6) .. (281.4,96.6) .. controls (261.4,116.6) and (213.4,118.6) .. (234.4,163.6) .. controls (255.4,208.6) and (109.4,196.6) .. (118.4,134.6) .. controls (127.4,72.6) and (171.13,102.27) .. (189.13,102.27) -- cycle ;
\draw [color={rgb, 255:red, 155; green, 155; blue, 155 }  ,draw opacity=1 ] [dash pattern={on 4.5pt off 4.5pt}]  (78.05,132.22) -- (309.05,133.22) ;
\draw [color={rgb, 255:red, 74; green, 144; blue, 226 }  ,draw opacity=1 ]   (130.05,201.22) -- (151.05,201.22) ;
\draw [color={rgb, 255:red, 74; green, 144; blue, 226 }  ,draw opacity=1 ]   (130.05,201.22) .. controls (120,201.3) and (122,220.3) .. (121.05,230.22) ;
\draw [color={rgb, 255:red, 74; green, 144; blue, 226 }  ,draw opacity=1 ]   (160.05,232.22) .. controls (160,220.3) and (161,201.3) .. (151.05,201.22) ;
\draw [color={rgb, 255:red, 74; green, 144; blue, 226 }  ,draw opacity=1 ]   (112.05,261.22) .. controls (120,260.3) and (121,239.3) .. (121.05,230.22) ;
\draw [color={rgb, 255:red, 74; green, 144; blue, 226 }  ,draw opacity=1 ]   (160.05,232.22) .. controls (160,239.3) and (160,260.3) .. (167.55,260.22) ;
\draw [color={rgb, 255:red, 208; green, 2; blue, 27 }  ,draw opacity=1 ]   (221.8,200.97) .. controls (229.8,200.97) and (231.05,222.22) .. (231.05,229.22) ;
\draw [color={rgb, 255:red, 208; green, 2; blue, 27 }  ,draw opacity=1 ]   (279.05,231.22) .. controls (279.05,217.22) and (282.05,199.22) .. (293.05,200.22) ;
\draw [color={rgb, 255:red, 208; green, 2; blue, 27 }  ,draw opacity=1 ]   (240.05,260.22) .. controls (230.05,259.22) and (231.05,222.22) .. (231.05,229.22) ;
\draw [color={rgb, 255:red, 208; green, 2; blue, 27 }  ,draw opacity=1 ]   (240.05,260.22) -- (270.05,260.22) ;
\draw [color={rgb, 255:red, 208; green, 2; blue, 27 }  ,draw opacity=1 ]   (279.05,231.22) .. controls (279.05,221.22) and (280.05,260.22) .. (270.05,260.22) ;
\draw [color={rgb, 255:red, 155; green, 155; blue, 155 }  ,draw opacity=1 ][line width=0.75]    (58.05,230.22) -- (326.05,231.21) ;
\draw [shift={(328.05,231.22)}, rotate = 180.21] [color={rgb, 255:red, 155; green, 155; blue, 155 }  ,draw opacity=1 ][line width=0.75]    (10.93,-3.29) .. controls (6.95,-1.4) and (3.31,-0.3) .. (0,0) .. controls (3.31,0.3) and (6.95,1.4) .. (10.93,3.29)   ;
\draw [color={rgb, 255:red, 155; green, 155; blue, 155 }  ,draw opacity=1 ][line width=0.75]    (58.05,230.22) -- (59.02,163.22) ;
\draw [shift={(59.05,161.22)}, rotate = 90.83] [color={rgb, 255:red, 155; green, 155; blue, 155 }  ,draw opacity=1 ][line width=0.75]    (10.93,-3.29) .. controls (6.95,-1.4) and (3.31,-0.3) .. (0,0) .. controls (3.31,0.3) and (6.95,1.4) .. (10.93,3.29)   ;
\draw [color={rgb, 255:red, 74; green, 144; blue, 226 }  ,draw opacity=1 ]   (59.55,260.72) -- (112.05,261.22) ;
\draw [color={rgb, 255:red, 74; green, 144; blue, 226 }  ,draw opacity=1 ]   (167.55,260.22) -- (342.55,260.22) ;
\draw [color={rgb, 255:red, 208; green, 2; blue, 27 }  ,draw opacity=1 ]   (59.3,201.47) -- (221.8,200.97) ;
\draw [color={rgb, 255:red, 208; green, 2; blue, 27 }  ,draw opacity=1 ]   (293.05,200.22) -- (326.05,200.22) ;
\draw [color={rgb, 255:red, 0; green, 0; blue, 0 }  ,draw opacity=1 ] [dash pattern={on 3.75pt off 3pt on 7.5pt off 1.5pt}]  (130.05,201.22) .. controls (120,200.3) and (117.08,220.62) .. (116.07,230.93) ;
\draw [color={rgb, 255:red, 0; green, 0; blue, 0 }  ,draw opacity=1 ] [dash pattern={on 3.75pt off 3pt on 7.5pt off 1.5pt}]  (116.07,230.93) .. controls (115.08,239.62) and (114.07,260.93) .. (106.07,260.93) ;
\draw [color={rgb, 255:red, 0; green, 0; blue, 0 }  ,draw opacity=1 ] [dash pattern={on 3.75pt off 3pt on 7.5pt off 1.5pt}]  (240.05,260.22) .. controls (231.08,261.72) and (231.05,242.22) .. (230.05,230.22) ;
\draw [color={rgb, 255:red, 0; green, 0; blue, 0 }  ,draw opacity=1 ] [dash pattern={on 6pt off 3pt on 7.5pt off 1.5pt}]  (230.05,230.22) .. controls (230.13,218.97) and (227.13,199.97) .. (221.8,200.97) ;
\draw  [dash pattern={on 3.75pt off 3pt on 7.5pt off 1.5pt}]  (130.05,201.22) -- (221.8,200.97) ;
\draw  [dash pattern={on 3.75pt off 3pt on 7.5pt off 1.5pt}]  (26,260.2) -- (106.07,260.93) ;
\draw  [dash pattern={on 3.75pt off 3pt on 7.5pt off 1.5pt}]  (240.05,260.22) -- (337,260.2) ;

\draw (41,219.4) node [anchor=north west][inner sep=0.75pt]    {$ \begin{array}{l}
\textcolor[rgb]{0.61,0.61,0.61}{0}\\
\end{array}$};
\draw (247,148.4) node [anchor=north west][inner sep=0.75pt]    {$\textcolor[rgb]{0.82,0.01,0.11}{\Psi }$};
\draw (129,143.4) node [anchor=north west][inner sep=0.75pt]    {$\textcolor[rgb]{0.29,0.56,0.89}{\Phi }$};
\draw (78,238.4) node [anchor=north west][inner sep=0.75pt]    {$u_{0}$};
\draw (169,76) node [anchor=north west][inner sep=0.75pt]    {$\Gamma _{0} = \partial \Omega_0$};
\draw (66,178.4) node [anchor=north west][inner sep=0.75pt]    {$\textcolor[rgb]{0.82,0.01,0.11}{\psi }$};
\draw (168,238.4) node [anchor=north west][inner sep=0.75pt]    {$\textcolor[rgb]{0.29,0.56,0.89}{\phi }$};

\end{tikzpicture}
\caption{Schematic illustration of the construction of initial conditions $u_0, \psi, \phi$.}
\label{fig:init_cond}
\end{figure}

\begin{proposition}[Maximal subsolution, minimal supersolution]\label{the:max_min_sub_super}
  Let $u_0 \in W^{2,\infty}(\dom)$ and $\phi,\psi \in W^{2,1}_\infty(\dom \times [0, \infty))$ be given and let $u$ be the unique viscosity solution from Theorem~\ref{the:existence}. For initial data $\mathbbold{1}_{\Omega_0} - \mathbbold{1}_{\Omega_0^c}$ and obstacles $\phi^{\Phi} := \sign(\sdist(\cdot,\Phi))$, $\psi^\Psi := -\sign(\sdist(\cdot,\Psi))$, the biggest viscosity subsolution of~\eqref{eq:level_set_equation} with $u_0, \phi^\Phi, \psi^\Psi$ is $\sign^*(u)$ and the smallest viscosity supersolution is $\sign_*(u)$, i.e., for any viscosity subsolution $v$ and any viscosity supersolution $w$ it holds
  \begin{align*}
    v(x,t) \leq \sign^*(u(x,t)) \quad \quad \text{for } (x,t) \in \dom \times (0, \infty),\\
    w(x,t) \geq \sign_*(u(x,t)) \quad \quad \text{for } (x,t) \in  \dom \times (0, \infty).
  \end{align*}
\end{proposition}
\begin{proof}
 For $\alpha \in \R$ and $\eps >0$, set 
 \begin{equation*}
  u^\eps(x,t) := \tanh((u(x,t)-\alpha)/\eps).
 \end{equation*}
 Then, because of the geometric property~\eqref{def:geometric}, $u^\eps$ is a viscosity solution of~\eqref{eq:level_set_equation} with obstacles $\phi^\eps(x) = \tanh((\phi(x)-\alpha)/\eps)$ and $\psi^\eps(x) = \tanh((\psi(x)-\alpha)/\eps)$.

 Theorem A.2 of the Appendix of~\cite{zbMATH04024130} (applied to the Hamiltonian of Remark~\ref{rem:hamiltonian}) states stability of viscosity solutions of arbitrary discontinuous Hamiltonians. In our case this yields that $u_\alpha(x,t) := \limsup_{\eps \downarrow 0, (y,s) \rightarrow (x,t)} u^\eps(y,s)$ is a viscosity subsolution of~\eqref{eq:level_set_equation} with obstacles $\phi_\alpha(x) := \limsup_{\eps \downarrow 0, y \rightarrow x} \phi^\eps(y)$ and $\psi_\alpha(x) := \limsup_{\eps \downarrow 0, y \rightarrow x} \psi^\eps(y)$. The properties of the hyperbolic tangent yield 
\begin{align*}
  u_\alpha(x,t) &= \begin{cases}
    1 & \text{if } u(x,t) > \alpha \text{ or } (x,t) \in \partial\{u > \alpha \},\\
    -1 & \text{if } u(x,t) < \alpha,\\
    0 & \text{if } (x,t) \in \Int\{u = \alpha\} \text{ or } (x,t) \in \partial\{u < \alpha \}\setminus \partial\{u > \alpha\},
  \end{cases}\\
  \phi_\alpha(x) &= \begin{cases}
    1 & \text{if } \phi(x) > \alpha \text{ or } x \in \partial\{\phi > \alpha \},\\
    -1 & \text{if } \phi(x) < \alpha,\\
    0 & \text{if } x \in \Int\{\phi = \alpha\} \text{ or } x \in \partial\{\phi < \alpha \}\setminus \partial\{\phi > \alpha\},
  \end{cases}\\
  \psi_\alpha(x) &= \begin{cases}
    1 & \text{if } \psi(x) > \alpha \text{ or } x \in \partial\{\psi > \alpha \},\\
    -1 & \text{if } \psi(x) < \alpha,\\
    0 & \text{if } x \in \Int\{\psi = \alpha\} \text{ or } x \in \partial\{\psi < \alpha \}\setminus \partial\{\psi > \alpha\}.
  \end{cases}
\end{align*}
A similar result holds for the supersolution. Apply the stability of~\cite{zbMATH04024130} once more to obtain that  $u^\infty(x,t) = \limsup_{\alpha \uparrow 0, (y,t) \rightarrow (x,t) } u_\alpha (y,t)$ is a subsolution with obstacles $\phi^\infty(x) = \limsup_{\alpha \uparrow 0, y \rightarrow x} \phi_\alpha(y)$ and $\psi^\infty(x) = \limsup_{\alpha \uparrow 0, y \rightarrow x} \psi_\alpha(y)$. Furthermore, one has 
\begin{equation*}
  u^\infty = \sign^*(u), \quad \text{ and } \quad  \phi^\infty = \sign^*(\phi) = (\phi^\Phi)^*,
\end{equation*} 
which implies that $u^\infty$ has initial conditions $\mathbbold{1}_{\Omega_0} - \mathbbold{1}_{\Omega_0^c}$ and respects the obstacles $\phi^\Phi$ and $\psi^\Psi$. 

So far we have shown that $\sign^*(u)$ is a subsolution; it remains to prove that it is the biggest subsolution by applying the comparison principle~\cite[Theorem 2.2]{zbMATH06841741}. To this end, take a sequence of smooth functions $\{\xi_n\}_n$ such that $\xi_n \equiv 1$ on $[0, \infty), \xi'_n \geq 0$ in $\R$, $\xi_n(\R) \subset [-1,1]$ and $\inf_n\xi_n = -1$ on $(-\infty,0)$. As the generic subsolution $v$ has initial values $\mathbbold{1}_{\Omega_0} - \mathbbold{1}_{\Omega_0^c}$ it follows $v^*(x,0) \leq \xi_n(u_0(x))$. Property~\eqref{def:geometric} ensures that $\xi_n(u)$ is a viscosity solution (so in particular a supersolution) with initial data $\xi_n(u_0) \geq v^*(\cdot,0)$ and obstacles $\xi_n(\phi), \xi_n(\psi)$. Note that, by Remark~\ref{rem:solution}, $v$ is also a subsolution with respect to the obstacles $\xi_n(\phi), \xi_n(\psi)$ as it holds $\xi_n \geq \sign$ which implies
\begin{align*}
   \phi^{\Phi} = -\sign(\sdist(\Phi)) = \sign(\phi) \leq \xi_n(\phi) \quad \text{ and } \quad \psi^{\Psi} = \sign(\sdist(\Psi)) = \sign(\psi) \leq \xi_n(\psi).
\end{align*}
So, for every $n \in \N$, the comparison principle~\cite[Theorem 2.2]{zbMATH06841741} is applicable. Hence, we conclude 
\begin{equation*}
  v \leq v^* \leq \inf_n \xi_n(u) = \sign^*(u).
\end{equation*}
The proof for the statement regarding the supersolution $v$ is omitted, but very similar.
\end{proof}

For the next proposition we will use the notation $\limsup^* u_h := \limsup_{(y, \ell h) \rightarrow (x,t)} u_h(y, \ell h)$ and $\liminf_* u_h := \liminf_{(y, \ell h) \rightarrow (x,t)} u_h(y, \ell h)$, where we implcitely mean that $h \downarrow 0$ in the limits.
\begin{proposition}\label{prop:u_hsubsuper}
  The functions $\limsup^*u_h$ and $\liminf_* u_h$ are, respectively sub- and supersolutions of~\eqref{eq:level_set_equation} with initial conditions $\mathbbold{1}_{\Omega_0} - \mathbbold{1}_{\Omega_0^c}$ and obstacles $\phi^\Phi, \psi^\Psi$.
\end{proposition}
\begin{proof}
This proof is basically exactly like in~\cite{zbMATH00763426}; we sketch the proof for $\bar{u} := \limsup_h^* u_h$ for the convenience of the reader. 

Let a test function $\eta \in C^2(\dom \times [0, \infty))$ and $(x,t) \in \dom \times (0, \infty)$ a strict global maximum of $\bar{u} - \eta$ be given. As $\bar{u} \in \{-1,1\}$ is upper semi-continuous, it follows from $\bar{u}(x,t) = -1$ that $\bar{u} \equiv -1$ in a neighborhood around $(x,t)$. Hence, 
\begin{equation*}
  D \eta (x,t) = 0, \quad D^2\eta(x,t) \geq 0, \quad \partial_t \eta (x,t) = 0,
\end{equation*}
holds true which immediately implies 
\begin{equation*}
  \partial_t \eta (x,t) + F_*(D\eta(x,t), D^2 \eta(x,t)) = \partial_t \eta (x,t) - \Delta \eta(x,t) + \lambda_{\text{min}}(D^2 \eta(x,t)) \leq 0.
\end{equation*}
Similarly, one reasons in the interior of the set $\{\bar{u} = 1\}$. Furthermore, by the definition of the scheme~\eqref{def:scheme} it immediately follows that $(\phi^\Phi)^* \leq \bar{u} \leq (\psi^\Psi)^*$ and $\bar{u} = \phi^\Phi$ on $\Phi$ as well as $\bar{u} = \psi^\Psi$ on $\Psi$. Therefore, one can assume that $(x,t)$ is in $\partial\{\bar{u} = 1\} \setminus( \Psi \cup \Phi)$. On top of that it also holds $\bar{u} = (\phi^{\Phi})^* = (\psi^\Psi)^* = 1$ on $\partial \Phi \cup \partial \Psi$ such that we can even assume $(x,t)$ is in $\partial\{\bar{u} = 1\} \setminus( \bar{\Phi} \cup \bar{\Psi})$.

One starts by noticing that $\limsup^* u_h = \limsup^* u_h^*$ so that one can use the latter. Using~\cite[Lemma A.3]{zbMATH04024130} and by assuming w.l.o.g.~$\liminf_{|x|, t \rightarrow \infty} \eta(x,t) = \infty$, there exists  a subsequence $(x_h, \ell_h h) \rightarrow (x,t)$ with 
\begin{equation*}
  u_h^*(x_h, \ell_h h) - \eta(x_h, \ell_h h) = \max_{\dom \times h\N} (u_h^* - \eta), \quad \quad \text{ and } \quad \quad u_h^*(x_h, \ell_h h ) \rightarrow u^*(x,t) = 1.
\end{equation*}
Due to $u_h \in \{-1,1\}$ the latter implies $u_h^*(x_h, \ell_h h) = 1$ for $h$ small enough and at the maximum point one thus has
\begin{equation*}
  u_h^*(x, \ell h) \leq 1 - \eta(x_h, \ell_h h) + \eta(x, \ell h) \quad \quad \text{ for any } (x,\ell) \in \dom \times \N.
\end{equation*}
Checking both cases $u_h^*(x, \ell h) = \pm 1$ one justifies that this implies 
\begin{equation*}
  u_h^*(x, \ell h) \leq \sign^*(\eta(x, \ell h ) - \eta (x_h, \ell_h h)) \quad \quad \text{ for any } (x, \ell) \in \dom \times \N.
\end{equation*}
Because $\Omega\setminus(\bar{\Phi} \cup \bar{\Psi})$ is open it follows from $x \in \Omega\setminus(\bar{\Phi} \cup \bar{\Psi})$ that we can assume w.l.o.g.~$x_h \in  \Omega\setminus(\bar{\Phi} \cup \bar{\Psi})$. The update by the scheme~\eqref{def:scheme} for any $x \in \Omega\setminus(\bar{\Phi} \cup \bar{\Psi})$ is 
\begin{equation*}
  u_h(x, \ell_h h) = \sign_*(e^{h \Delta}u_h (x, (\ell_h - 1)h)),
\end{equation*}
which immediately implies 
\begin{equation*}
    u_h^*(x, \ell_h h) \leq \sign^*(e^{h \Delta}u_h^* (x, (\ell_h - 1)h)).
\end{equation*}
Assembling the previous steps for $x = x_h$ yields 
\begin{align*}
  1 = u_h^*(x_h, \ell_h h) &\leq \sign^*(e^{h \Delta}u_h^* (x_h, (\ell_h - 1)h)) \\
  &\leq \sign^*\big(e^{h \Delta} \big[\sign^*(\eta(\cdot, (\ell_h-1) h ) - \eta (x_h, \ell_h h))\big](x_h)\big)
\end{align*}
which is equivalent to 
\begin{equation*}
  e^{h \Delta} \big[\sign^*(\eta(\cdot, (\ell_h-1) h ) - \eta (x_h, \ell_h h))\big](x_h) \geq 0.
\end{equation*}
As $e^{h \Delta} 1 = 1$ this can again be reformulated as 
\begin{equation*}
  \frac{1}{(4\pi h)^{d/2}} \int_{\{\eta(y, (\ell -1)h) - \eta(x_h, \ell_h h) \geq 0 \}} e^{-\frac{|x_h - y|^2}{4h}} \, dy \geq \frac{1}{2}.
\end{equation*}
For the case $D\eta(x,t) \neq 0$ one applies the consistency result~\cite[Proposition 4.1]{zbMATH00763426} with $\eta_h(x,t ) = \eta(x,t) - \eta(x_h, \ell_h h)$ to conclude with~\eqref{eq:lowerF} that
\begin{align*}
  0 &\geq 2\sqrt{\pi} |D\eta(x,t)| \liminf_{h \downarrow 0} \left(\frac{1}{2} - \frac{1}{(4\pi h)^{d/2}} \int_{\eta_h(y, (\ell -1)h) \geq 0} e^{- \frac{|x_h - y|^2}{4h}} \, dy \right)\\
  &\geq \partial_t\eta(x,t) + F_*(D\eta(x,t), D^2 \eta(x,t)).
\end{align*}
Furthermore, if $D \eta(x,t) = 0$, by~\cite[Proposition 2.2]{zbMATH00763426} it is enough to check the case $D^2 \eta(x,t)= 0$. By~\cite[Proposition 4.1]{zbMATH00763426} it holds $\partial_t \eta(x,t) \leq 0$, so that indeed
\begin{equation*}
  0 \geq \partial_t \eta(x,t)= \partial_t \eta(x,t)  + F_*(D\eta(x,t), D^2 \eta(x,t)).
\end{equation*}
As the obstacles constraints are satisfied, this finishes the proof for the subsolution.
\end{proof}

Having done the main work in the two previous propositions, the main result of this section follows immediately. 
\begin{theorem}\label{the:main}
  Let $\Omega_t := \{x \in \dom: u(x,t) > 0\}$ and $\Gamma_t := \{x \in \dom: u(x,t) = 0\}$, then, for any $t> 0$ it holds
  \begin{equation*}
    {\liminf}_* u_h(x,t) = 1 \quad \text{ in } \Omega_t, \quad \text{ and } \quad {\limsup}^*u_h(x,t) = -1 \quad \text{ in } (\Omega_t \cup \Gamma_t)^c .
  \end{equation*}
\end{theorem}
\begin{proof}
  Combing Proposition~\ref{the:max_min_sub_super} with Proposition~\ref{prop:u_hsubsuper} immediately yields
  \begin{equation*}
    {\limsup}^* u_h \leq \sign^*(u) \quad \text{ and } \quad {\liminf}_* u_h \geq \sign_*(u) \text{ in } \dom \times (0, +\infty).
  \end{equation*}
  This implies the result as $u_h$ has values in $\{-1,1\}$.
\end{proof}

The next result is basically exactly as in \cite[Corollary 1.3]{zbMATH00763426}. One only needs to care about compactness in time, as the flow might not converge to a steady state in final time.
\begin{corollary}\label{cor:conv_front}
  For any finite time horizon $T > 0$, if $\Omega_0$ is bounded and $\cup_{t = 0}^T \Gamma_t \times \{t\} = \partial \{(x,t): u(x,t) > 0\} = \partial \{(x,t): u(x,t) < 0\}$, then $\Sigma^h := \cup_{\ell = 0}^{\lfloor\frac{T}{h}\rfloor} \partial\{x \in \dom: u_h(x, \ell h) = 1\} \times \{\ell h\}$ converges to $\Sigma := \cup_{t = 0}^T \Gamma_t \times \{t\}$ in the Hausdorff distance. 
\end{corollary}
To prove Corollary \ref{cor:conv_front} we will use the following lemma about boundedness of the iterates of the scheme. 
\begin{lemma}\label{lem:its_stay_bounded}
  If $\Omega_0 \subset \bar{B}_R$ for any given radius $R > 0$ then $\Omega_{\ell h}^h := \{x \in \dom: u_h(x, \ell h) = 1\} \subset \bar{B}_R$ for all $\ell \in \N$.
\end{lemma}
\begin{proof}
  We first notice that $\Omega_0 \subset \bar{B}_R$ implies $\Phi \subset \bar{B}_R$. If one initializes the scheme with a half space $H_e := \{x \in \dom: e \cdot x < R\}$ for some $e \in \S^{d-1}$, i.e.,
  \begin{equation*}
    u_h^e(x,0) = - \mathbbold{1}_\Psi \wedge (\mathbbold{1}_{H_e} - \mathbbold{1}_{H_e^c}),
  \end{equation*} 
  then, using $e^{h \Delta} (- \mathbbold{1}_\Psi \wedge (\mathbbold{1}_{H_e} - \mathbbold{1}_{H_e^c})) \leq e^{h \Delta} (\mathbbold{1}_{H_e} - \mathbbold{1}_{H_e^c})$, one inductively checks that 
  \begin{equation}\label{eq:mon_half_space}
        u_h^e(\cdot,\ell h) \leq - \mathbbold{1}_\Psi \wedge (\mathbbold{1}_{H_e} - \mathbbold{1}_{H_e^c}) \quad \quad \text{for all } \ell \in \N.
  \end{equation}
  Furthermore, as the scheme is order preserving and $\Omega_0$ respects the obstacles it follows from
  \begin{equation*}
    \mathbbold{1}_{\Omega_0} - \mathbbold{1}_{\Omega_0^c} \leq - \mathbbold{1}_\Psi \wedge (\mathbbold{1}_{H_e} - \mathbbold{1}_{H_e^c}),
  \end{equation*}
  together with \eqref{eq:mon_half_space} that the inequality is preserved for all later times $\ell \in \N$, i.e.,
  \begin{equation*}
    u_h(\cdot, \ell h) \leq u_h^e(\cdot, \ell h) \leq  - \mathbbold{1}_\Psi \wedge (\mathbbold{1}_{H_e} - \mathbbold{1}_{H_e^c}).
  \end{equation*}
  As $e \in \S^{d-1}$ was arbitrary, one concludes the lemma.
\end{proof}

\begin{proof}[Proof of Corollary \ref{cor:conv_front}]
  By the definition of the Hausdorff distance we need to prove that 
  \begin{equation*}
    0 \xleftarrow{h \downarrow 0} \dist( \Sigma,  \Sigma^h):= \max\{ \max_{(x,t) \in  \Sigma^h} \dist((x,t),  \Sigma), \max_{(x,t)\in \Sigma} \dist((x,t),  \Sigma^h) \}.
  \end{equation*}
  We start by proving $\max_{(x,t) \in  \Sigma^h} \dist((x,t),  \Sigma)\rightarrow 0$. Let $R > 0$ be such that $\Omega_0 \subset \bar{B}_R$ (which exists as $\Omega_0$ is bounded by assumption). For fixed but arbitrary $\eps > 0$ define $K_{\eps} := \{(x,t) \in \bar{B}_R\setminus(\Psi \cup \Phi) \times [0, T]: \dist((x,t),  \Sigma) \geq \eps\}$. The function $(x,t) \mapsto \sign(u(x,t))$ is continuous on $K_{\eps}$ since it is locally constant on $K_\eps$. Hence, on the compact set $K_\eps$ (because $\Phi$ and $\Psi$ are open by assumption) it holds by Proposition~\ref{prop:u_hsubsuper} and Proposition~\ref{the:max_min_sub_super} that
  \begin{equation*}
    {\liminf}_* u_h \leq {\limsup}^* u_h \leq \sign^*(u) = \sign_*(u) \leq {\liminf}_* u_h,
  \end{equation*}
  which implies
  \begin{equation}\label{eq:supinfsign}
     {\liminf}_* u_h = {\limsup}^* u_h = \sign(u).
  \end{equation}
  Assume that $u_h$ does not convergence uniformly to $\sign(u)$ on $K_\eps$. Then there exists a subsequence $(x_h, t_h)$ and $\delta > 0$ such that $|\sign(u(x_h,t_h)) - u_h(x_h,t_h)| > \delta$. As $K_\eps$ is compact it holds $(x_h, t_h) \rightarrow (x,t)$ for some $(x,t) \in K_\eps$ up to a further subsequence. But this already implies by~\eqref{eq:supinfsign} and the continuity of $u$ that $0 = |\sign(u(x,t)) - \sign(u(x,t))| \geq \delta$, a contradiction. Hence, $u_h$ has the uniform limit $\sign(u)$ on $K_\eps$. The uniform convergence implies $u_h = \sign(u)$ for $h$ small enough on $K_\eps$ as $u_h$ is $\pm 1$-valued. This means that $u_h$ is locally constant on $K_\eps$ for $h$ small enough which means that $\Sigma_h$ cannot intersect the interior of $K_\eps$. One concludes that $\Sigma_h \subset \{(x,t) \in \bar{B}_R \times [0, T]: \dist((x,t),  \Sigma) \leq \eps\}$. As $\eps > 0$ was arbitrary, this shows the claimed $\max_{(x,t) \in \Sigma^h} \dist((x,t), \Sigma) \rightarrow 0$.

  Turning towards the second claim $ \max_{(x,t)\in \Sigma} \dist((x,t), \Sigma^h) \rightarrow 0$, one notices that the assumed $\cup_{t = 0}^T \Gamma_t \times \{t\} = \partial \{(x,t): u(x,t) > 0\} = \partial \{(x,t): u(x,t) < 0\}$ implies
  \begin{equation*}
    (\sign^*(u))_* = \sign_*(u) \quad \quad \text{ and } \quad \quad (\sign_*(u))^* = \sign^*(u) \quad \text{ in } \dom \times (0, T).
  \end{equation*}
  Furthermore, by Proposition~\ref{prop:u_hsubsuper} we know that 
  \begin{equation*}
    \bar{u}(x,t) = \limsup_{y \rightarrow x, \ell h \rightarrow t} u_h(y, \ell h) \quad \quad \text{ and } \quad \quad \ubar{u}(x,t) = \liminf_{y \rightarrow x, \ell h \rightarrow t} u_h(y , \ell h),
  \end{equation*}
  are viscosity sub- and supersolution, respectively, with initial data $\mathbbold{1}_{\Omega_0} - \mathbbold{1}_{\Omega_0^c}$ and obstacles $\phi^\Phi, \psi^\Psi$. Again, Proposition~\ref{the:max_min_sub_super} implies now
  \begin{equation*}
    \sign_*(u(x,t)) \leq \ubar{u}(x,t) \leq \bar{u}(x,t) \leq \sign^*(u(x,t)).
  \end{equation*}
  Application of the upper semi-continuous envelope or lower semi-continuous envelope yields
  \begin{equation*}
    \liminf_{y \rightarrow x,\ \ell h \rightarrow t} u_h(y, \ell h) = \sign_*(u(x,t)) \quad \quad \text{ and } \quad \quad \limsup_{y \rightarrow x,\ \ell h \rightarrow t} u_h(y, \ell h) = \sign^*(u(x,t)).
  \end{equation*}
  Assume there is a sequence $h \downarrow 0$ such that $\max_{(x,t)\in \Sigma} \dist((x,t),  \Sigma^h)  \geq \eps >0$. Because $\Sigma$ is compact in $\dom \times [0, T]$, there exists a subsequence, again denoted by $h$, and points $(x_h, t_h) \in \Sigma$ converging to some $(x,t) \in \Sigma$ such that $\dist((x_h, t_h), \Sigma^h) \geq \eps$. Again up to a subsequence we therefore have either 
  \begin{equation*}
    u_h(y, \ell h) = 1 \quad \quad \text{ for any } h \text{ if } |y - x_h| + |\ell h - t_h| < \eps
  \end{equation*}
  or 
    \begin{equation*}
    u_h(y, \ell h) = -1 \quad \quad \text{ for any } h \text{ if } |y - x_h| + |\ell h - t_h| < \eps.
  \end{equation*}
  In the first case this implies 
  \begin{equation*}
    \sign_*(u(x,t)) = \liminf_{y \rightarrow x, \ell h\rightarrow t} u_h(y, \ell h) = 1. 
  \end{equation*}
  But this is a contradiction as $(x,t) \in \Sigma$ which means $\sign_*(u(x,t)) = -1$. Similarly, one concludes the other case. 
\end{proof}

\section{Space Discretization}\label{sec:convergence_space}
We discretize our domain $\dom$ into finitely many points $X_N \subset \dom$ with $\# X_N = N$. In this numerical setting $\Omega$ has to be finite with some sort of boundary conditions to be computable. So, choosing $\Omega = \mathbb{T}^d$ as the $d$-dimensional torus is a very natural choice. Denote by $\mu_N := \frac{1}{N} \sum_{x \in X_N} \delta_x$ the empirical measure of $X_N$. Furthermore assume there are functions $T_N: \dom \rightarrow X_N$ such that $T_N \rightarrow \id$ in $L^\infty(\dom)$ and $\mu_N = (T_N)_{\#}\mu$ for some measure $\mu = \rho \, \mathcal{L}^d\big|_{\dom}$ that has density $\rho > 0$ with respect to the Lebesgue meausure on $\dom$. Let $\Delta_N:\{f: X_N \rightarrow \R\} \rightarrow \{f: X_N \rightarrow \R\}$ be a symmetric, linear discretization of the Laplacian. We assume that $\Delta_N$ is diagonal dominant, i.e., $(\Delta_N)_{ii} + \sum_{j \neq i} |(\Delta_N)_{ij}| \leq 0$ for any $i = 1, \dots, N$. Note that this is equivalent to the maximum principle
\begin{equation}\label{eq:disc_max_principle}
  u \in [-1,1] \Rightarrow e^{h \Delta_N} u \in [-1,1]
\end{equation}
for any $h \geq 0$.
Furthermore, we assume the discrete-to-continuum convergence in the form 
\begin{equation}\label{eq:weaktostrong}
  (e^{h \Delta_N} \schvar_N) \circ T_N \rightarrow e^{h \Delta}\schvar \quad \quad \text{ in } L^2(\dom) \text{ if } \schvar_N \circ T_N \rightharpoonup \schvar \text{ weakly in } L^2(\dom).
\end{equation}
The convergence in~\eqref{eq:weaktostrong} is the so called (weak) $TL^2$-convergence; for more information on this topology see~\cite{zbMATH06559722}. Recall that we use the sign convention of the Laplacian such that $\Delta_N$ has negative eigenvalues.

\begin{rem}\label{rem:space_disc}
  The above framework covers the standard discretization scenarios, for example a uniform grid or a random geometric graph. 
  \begin{itemize}
    \item[i)] Under a random geometric graph we understand that $X_N$ is a collection of i.i.d.~samples of $\mu = \rho \mathcal{L}^d$ and 
    \begin{equation*}
      \Delta_N u (x) = -\frac{1}{N} \sum_{y \in X_N} w_\eps(x,y) \left(\frac{u(x) - u(y)}{ \eps^2}\right)
    \end{equation*}
    is a graph Laplacian with weights $w_\eps(x,y) = \frac{1}{\eps^d}\eta(\frac{|x-y|}{\eps})$ for a suitable kernel $\eta: \R_+ \rightarrow \R_+$ and a length scale $\eps$. For further information see for example~\cite{zbMATH08091182}. One can then define optimal transport maps $T_N$ as in \cite{zbMATH08091182} to get the convergence~\eqref{eq:weaktostrong} with probability $1$ as done in \cite[Theorem 2]{zbMATH08091182} or in the best possible parameter regime in \cite[Theorem 4.3]{arXiv:2410.07776}.
    \item[ii)] Also deterministic discretizations fit in the framework. One needs to choose $X_N$ such that the empirical measure of $X_N$ converges to a measure $\mu$. This is for example the case if $X_N = \T^d \cap \eps \mathbb{Z}^d$ is a uniform grid. One could then take
    \begin{equation*}
      T_N: \T^d \rightarrow X_N, x \mapsto \eps z \text{ if } \|x - \eps z + Lk\|_{\ell^1} < \frac{\eps}{2} \text{ for some } k \in \Z^d \text{ and } z \in \mathbb{Z}
    \end{equation*}
    together with $\mu = \mathcal{L}^d$ and for any $u = (u_i)_{i \in X_N}$ the graph Laplacian induced by the kernel $\eta = \mathbbold{1}_{[0,1]}$ simplifies to a finite difference matrix 
    \begin{equation}\label{def:finite_dif_laplace}
      \big(\Delta_N u(x)\big)_i = \sum_{j = 1}^d \left(\frac{u_{i + \eps e_j} - 2u_i + u_{i-\eps e_j}}{ \eps^2}\right).
    \end{equation}
    One can show with the same techniques as in case (i) that~\eqref{eq:weaktostrong} holds.
  \end{itemize}
\end{rem}

We define the discrete obstacle MBO scheme on $X_N$ by 
\begin{equation}\label{def:discrete_scheme}
  \schvar_N^{\ell + 1} = \sign(e^{h\Delta_N}\schvar_N^\ell + \mathbbold{1}_{\Phi_N} - \mathbbold{1}_{\Psi_N})
\end{equation}
where $\Phi_N, \Psi_N \subset \dom$ are discrete obstacles. We will later also denote 
\begin{equation*}
  \phi_N(x) := \begin{cases}
    1 & \text{if } x \in \Phi_N\\
    -1 & \text{if } x \notin \Phi_N, 
  \end{cases}
  \quad \quad \text{ and }   \psi_N(x) := \begin{cases}
    1 & \text{if } x \in \Psi_N\\
    -1 & \text{if } x \notin \Psi_N. 
  \end{cases}
\end{equation*}

We want to prove the convergence of the discrete scheme~\eqref{def:discrete_scheme} to the continuous scheme~\eqref{def:scheme}. Our main tool to that aim will be the following minimizing movement interpretation that is an adaptation of~\cite{zbMATH06430104}.
\begin{lemma}\label{lem:minmov}
  The scheme~\eqref{def:discrete_scheme} is equivalent to 
  \begin{equation} \label{eq:minmov}
    \schvar_N^{\ell + 1} \in \argmin_{\schvar:X_N \rightarrow \{-1,1\}} E_N(\schvar) + \frac{1}{2h} d_N^2(\schvar, \schvar_N^\ell) \quad \quad \text{ s.t. } \schvar \equiv 1 \text{ on } \Phi_N \quad  \text{ and } \quad \schvar \equiv - 1 \text { on } \Psi_N,
  \end{equation}
  or reformulated with a Lagrange multiplier
  \begin{equation}\label{eq:minmovlag}
    \schvar_N^{\ell + 1} \in \argmin_{\schvar:X_N \rightarrow \{-1,1\}} E_N(\schvar) + \frac{1}{2h} d_N^2(\schvar, \schvar_N^\ell)  + \frac{4}{\sqrt{h}N} \sum_{x \in X_N} \Big( (\phi_N(x)  - u(x))_+ + (\psi_N(x) -u(x) )_- \Big)
  \end{equation}
  where the energy and distance are given by 
  \begin{align*}
  E_N(\schvar) &= \frac{1}{\sqrt{h}N} \sum_{x \in X_N} (1- \schvar)(x) \left(e^{h \Delta_N} (1 + \schvar)\right)(x),\\
  \frac{1}{2h}d_N^2(\schvar, \tilde{\schvar}) &= \frac{1}{\sqrt{h}N} \sum_{x \in X_N} (\schvar- \tilde{\schvar})(x)\left( e^{h \Delta_N} (\schvar - \tilde{\schvar}) \right)(x).
  \end{align*}
\end{lemma}

\begin{proof}
  We start with the equivalence of the scheme~\eqref{def:discrete_scheme} to \eqref{eq:minmov}. Towards that aim define the symmetric bilinear form 
  \begin{equation}\label{eq:undoing_square}
    b(f,g):=  \frac{1}{\sqrt{h}N} \sum_{x \in X_N} f(x) \left(e^{h \Delta_N}g\right)(x).
  \end{equation}
  Then it holds 
  \begin{align*}
    E_h(\schvar) + \frac{1}{2h} d_N^2(\schvar, \schvar_n^\ell) &= b(1- \schvar, 1 + \schvar) +  b(\schvar - \schvar_N^\ell, \schvar-\schvar_N^\ell) \\
    &= b(1,1) - 2b(\schvar, \schvar_N^\ell) + b(\schvar_N^\ell, \schvar_N^\ell)
  \end{align*}
  such that the minimization problem~\eqref{eq:minmov} is equivalent to
  \begin{equation*}
    \schvar_N^{\ell + 1} \in \argmin_{\schvar:X_N \rightarrow \{-1,1\}} - \frac{2}{\sqrt{h}N} \sum_{x \in X_N} \schvar(x) \left(e^{h \Delta_N} \schvar_N^\ell \right)(x) \quad \quad \text{ s.t. } \schvar \equiv 1 \text{ on } \Phi_N \text{ and } \schvar \equiv - 1 \text { on } \Psi_N.
  \end{equation*}
  The above problem can be minimized pointwise which yields that a minimizier satisfies
  \begin{equation*}
    \schvar(x) = \begin{cases}
      1 & \text{if } x\in \Phi_N \text{ or } (x \notin \Psi_N \text{ and } \left(e^{h \Delta_N} \schvar_N^\ell \right)(x) > 0),\\
      -1 & \text{if } x\in \Psi_N \text{ or } (x \notin \Phi_N \text{ and } \left(e^{h \Delta_N} \schvar_N^\ell \right)(x) < 0);
    \end{cases}
  \end{equation*}
  or equivalently the minimizer is given by~\eqref{def:discrete_scheme}.

   Next we prove that~\eqref{eq:minmov} is equivalent to~\eqref{eq:minmovlag}. For this it is enough to prove that a minimizer $\schvar_N^\ell$ of~\eqref{eq:minmovlag} satisfies $\phi_N \leq \schvar_N^\ell \leq \psi_N$; indeed this is enough because we then have that 
   \begin{equation*}
         0 = \frac{4}{\sqrt{h}N} \sum_{x \in X_N} (\phi_N(x)  - u(x))_+ + (\psi_N(x) -u(x) )_- 
   \end{equation*}
   holds and we can restrict the set of competitors in~\eqref{eq:minmovlag} to the admissible functions of \eqref{eq:minmov}.

    To prove the claim assume that $\schvar_N^{\ell+1}$ is given by~\eqref{eq:minmovlag}. We need to show that the $\phi_N \leq \schvar_N^{\ell+ 1} \leq \psi_N$ holds.
  We check $\phi_N \leq \schvar_N^{\ell + 1}$. The other inequality is treated similarly. So let $x \in \Phi_N$ be given. Expanding the square as in the lines following~\eqref{eq:undoing_square} yields that the minimization in~\eqref{eq:minmovlag} can be done pointwise and a minimizer has to satisfy
  \begin{align*}
    \schvar_N^{\ell+1} ({x}) \in \argmin_{z \in \{-1,1\}} - \frac{2}{\sqrt{h}N}  z \left(e^{h \Delta_N} \schvar_N^\ell \right)({x}) + \frac{4}{\sqrt{h}N}(\phi_N(x) - z)_+.
  \end{align*}
  Assume that $\phi_N({x}) >\schvar_N^{\ell+1} ({x})$ which implies that $\schvar_N^{\ell +1}({x}) = -1$ and $\phi_N({x}) = 1$ but this means 
  \begin{align*}
    -1 \in \argmin_{z \in \{-1,1\}} - z \left(\left(e^{h \Delta_N} \schvar_N^\ell \right) ({x}) + 2\right),
  \end{align*}
  which is a contradiction as $\left(e^{h \Delta_N} \schvar_N^\ell \right) ({x}) + 2 > 0$ by the discrete maximum principle~\eqref{eq:disc_max_principle}.
\end{proof}

The analogous statement holds also on the continuous level.
\begin{lemma}\label{lem:minmov_cont}
  The scheme~\eqref{def:scheme} is equivalent to 
  \begin{equation} \label{eq:minmov_cont}
    \schvar^{\ell + 1} \in \argmin_{\schvar\colon\dom \rightarrow \{-1,1\}} E_h(\schvar) + \frac{1}{2h} d_h^2(\schvar, \schvar_N^\ell) \quad \quad \text{ s.t. } \schvar \equiv 1 \text{ on } \Phi \quad \text{ and } \quad \schvar \equiv - 1 \text { on } \Psi, 
  \end{equation}
  or reformulated with a Lagrange multiplier
  \begin{equation}\label{eq:minmovlag_cont}
    \schvar^{\ell + 1} \in \argmin_{\schvar\colon\dom \rightarrow \{-1,1\}} E_h(\schvar) + \frac{1}{2h} d_h^2(\schvar, \schvar_N^\ell)  + \frac{4}{\sqrt{h}} \int_{\dom} (\phi(x)  - \schvar(x))_+ + (\psi(x) -\schvar(x) )_- \, d\mu
  \end{equation}
  where the energy and distance are given by~\eqref{def:energy}.
\end{lemma}

The next proposition is the main statement of this section and states that the space-discretized obstacle MBO scheme converges to the space-continuous obstacle MBO scheme.
\begin{proposition}\label{prop:mean_field_limit}
  Let $\schvar_N^0, \phi_N, \psi_N: X_N \rightarrow \{-1,1\}$ together with $\schvar^0, \phi^\Phi, \psi^\Psi: \dom \rightarrow \{-1,1\}$ be given such that $\schvar_N^0 \circ T_N \rightharpoonup \schvar^0, \phi_N \circ T_N \rightharpoonup \phi, \psi_N \circ T_N \rightharpoonup \psi$ weakly in $L^2(\dom)$. Let $\schvar_N^\ell:X_N \rightarrow \{-1,1\}$ be the iterates of~\eqref{def:discrete_scheme} with initial condition $\schvar_N^0$ and obstacles $\phi_N, \psi_N$ and $\schvar^\ell$ the iterates of~\eqref{def:scheme} with initial condition $\schvar^0$ and obstacles $\Phi, \Psi$. Then it holds up to subsequences
  \begin{equation*}
    \schvar_N^\ell \circ T_N \rightharpoonup \schvar^\ell \quad \quad\text{ weakly in } L^2(\dom) \text{ for any } \ell \in \N.
  \end{equation*}
\end{proposition}
\begin{rem}
  The condition $\phi_N \circ T_N \rightharpoonup \phi^\Phi$ is for example satisfied if one uses $\phi_N = \phi^\Phi\big|_{X_N}$. To see this one takes a test function $f \in L^2(\dom)$ and computes
  \begin{align*}
    \int_{\dom} (\phi_N \circ T_N - \phi^\Phi)f\, d\mu &= 2 \int_{\dom} \left(\mathbbold{1}_{\{\phi_N \circ T_N = 1\} \setminus \Phi} - \mathbbold{1}_{\{\phi_N \circ T_N = -1\} \setminus \Phi^c}\right) f \,d \mu\\
    &\leq 2 \left(2\int_{\dom}  \mathbbold{1}_{\{\phi_N \circ T_N = 1\} \setminus \Phi} + \mathbbold{1}_{\{\phi_N \circ T_N = -1\} \setminus \Phi^c} \right)^{\frac{1}{2}} \|f\|_{L^2}\\
    &\leq 2 ( 4\mathcal{H}^{d-1}(\partial\Phi) \ \omega(\|T_N - \id\|_{\infty}))^{\frac{1}{2}} \|f\|_{L^2} \\
    &\xrightarrow{N \rightarrow \infty} 0
  \end{align*}
  where $\omega(s)$ is the volume of a $d$-dimensional ball with radius $s$.
\end{rem}
Now we turn to the proof of Proposition \ref{prop:mean_field_limit}, which relies on $\Gamma$-convergence. Curiously, it turns out to be convenient to consider the functional $E_N(u) + \frac{1}{2h}d_N^2 (u, u_N^\ell)$ as a perturbation.
\begin{proof}[Proof of Proposition \ref{prop:mean_field_limit}]
  We prove the statement by showing that for every $\ell \in \N$
  \begin{equation*}
    G_N(\schvar) := E_N(\schvar) + \frac{1}{2h} d_N^2(\schvar, \schvar_N^\ell)  + \frac{4}{\sqrt{h}N} \sum_{x \in X_N} (\phi_N(x)  - \schvar_N(x))_+ + (\psi_N(x) -\schvar_N(x) )_- \, d\mu(x).
  \end{equation*}
  $\Gamma$-convergences in the weak $TL^2$ topology to the its continuous mean field limit
  \begin{equation*}
    G(\schvar) := E_h(\schvar) + \frac{1}{2h} d_h^2(\schvar, \schvar^\ell)  + \frac{4}{\sqrt{h}} \int_{\dom} (\phi(x)  - \schvar(x))_+ + (\psi(x) -\schvar(x) )_-. 
  \end{equation*}
  With weak $TL^2$ topology it is meant that $f_N:X_N \rightarrow \R$ converges to $f:\dom \rightarrow \R$ if $f_N \circ T_N \rightharpoonup f$ weakly in $L^2(\dom)$. $\Gamma$-compactness is trivial in this case, because $\schvar_N^\ell \circ T_N$ has values in $\{-1,1\}$. We prove the $\Gamma$-convergence for $\ell = 0$; for general $\ell \in \N$ the statement follows by induction.

  We first note that because of \eqref{eq:weaktostrong} it holds for any $\schvar_N: X_N \rightarrow \R$ with $\schvar_N \circ T_N \rightharpoonup \schvar$ that
  \begin{align*}
    E_N(\schvar_N) &= \frac{1}{\sqrt{h}N} \sum_{x \in X_N} (1- \schvar_N)(x) \left(e^{h \Delta_N} (1 + \schvar_N)\right)(x)\\
    &= \frac{1}{\sqrt{h}} \int_{\dom} (1- \schvar_N)(T_N(x)) \left(e^{h \Delta_N} (1 + \schvar_N)\right)(T_N(x)) \, d\mu(x)\\
    &\xrightarrow{N \rightarrow \infty} \frac{1}{\sqrt{h}} \int_{\dom} (1- \schvar)(x) \left(e^{h \Delta} (1 + \schvar)\right)(x) \, d\mu(x)\\
    &= E_h(\schvar).
  \end{align*}
  Similarly it holds
  \begin{equation}
    \frac{1}{2h}d_N^2(\schvar_N, \schvar_N^0) \xrightarrow{N \rightarrow \infty} \frac{1}{2h}d_h^2(\schvar, \schvar^0).
  \end{equation}
  Hence $E_N + \frac{1}{2h} d_N^2(\cdot, u_N^\ell)$ can be viewed as a continuous perturbation and by the stability of $\Gamma$-convergence it is enough to prove that 
  \begin{equation}
    \frac{4}{\sqrt{h}N} \sum_{x \in X_N} \Big( (\phi_N(x)  - \schvar_N(x))_+ + (\psi_N(x) -\schvar_N(x) )_- \Big)
  \end{equation}
  $\Gamma$-convergences to its mean field limit. We concentrate on the first summand; the other one is treated similarly. Starting with the lower bound let a sequence of functions $\schvar_N: X_N \rightarrow \R$ and $\schvar: \dom \rightarrow \R$ be given such that $\schvar_N \circ T_N \rightharpoonup \schvar$ weakly in $L^2(\dom)$. As $\phi_N \circ T_N \rightarrow \phi^\Phi$ weakly in $L^2(\dom)$ it is enough to prove that
  \begin{align*}
    \tilde{G}(f):=\frac{4}{\sqrt{h}} \int_{\dom} f_+(x) \, d\mu(x)
  \end{align*}
  is lower-semicontinuous with respect to the weak $L^2$-topology. But this is just a consequence of the continuity of $\tilde{G}$ in the strong $L^2$-topology together with the convexity of $\tilde{G}$.

  Turning towards the recovery sequence take any sequence of Lipschitz-approximations $\{\tilde{\schvar}_\delta\}_\delta$ such that $\tilde{\schvar}_\delta \xrightarrow{\delta \downarrow 0} \schvar$ strongly in $L^2(\dom \setminus (\Phi \cup \Psi))$ and $\tilde{\schvar}_\delta = 1$ on $\Phi$ and $\tilde{\schvar}_\delta = -1$ on $\Psi$. Because $G$ is continuous in $L^2(\dom)$ it is thus enough to construct a recovery sequence for $\tilde{\schvar}_\delta$ with $\delta$ fixed. The recovery sequence for $\schvar$ then follows by taking a diagonal sequence. We now define $\schvar_N$ as the restriction of $\tilde{\schvar}_\delta$ to the set $X_N$, i.e.,
  \begin{equation*}
    \schvar_N(x) := \tilde{\schvar}_\delta \big|_{X_N} (x).
  \end{equation*}
  By the Lipschitz-continuity of $\tilde{u}_\delta$ it follows
  \begin{equation*}
    \int_{\dom} |\schvar_N \circ T_N - \tilde{\schvar}_\delta| \, d\mu 
    \leq     \int_{\dom} |\schvar_N(T_N(x)) - \tilde{\schvar}_\delta(T_N(x))| + |\tilde{\schvar}_\delta(T_N(x)) - \tilde{\schvar}_\delta(x)| \, d\mu \leq |\dom| L \|T_N - \id\|_{\infty}.
  \end{equation*}
  Hence, using $T_N \rightarrow \id$ in $L^\infty(\dom)$ one concludes that $\schvar_N \circ T_N$ converges to $\tilde{\schvar}_\delta$ in $L^1(\mu)$ (and by boundedness also in $L^2(\mu)$) as $N\rightarrow \infty$. In the following we denote $\Phi_N := \{\phi_N \circ T_N  = 1\}$. Next we want to show that the symmetric difference $\Phi_N \triangle \Phi$ converges to zero. To that aim define the auxiliary functions
  \begin{equation*}
    f_N := \frac{\phi_N \circ T_N + 1}{2} \in \{0,1\}, \quad \quad \text{ and } \quad \quad f := \frac{\phi^\Phi+ 1}{2} \in \{0,1\}.
  \end{equation*}
  As $\phi_N \circ T_N \rightharpoonup \phi^\Phi$ it holds $f_N \rightharpoonup f$ weakly in $L^2(\dom)$ as $N \rightarrow \infty$. Therefore, it follows
  \begin{align*}
    |\Phi_N \triangle \Phi| = \int_{\dom} f(1-f_N) + f_N(1-f) \, d\mu \rightarrow \int_{\dom}f(1-f) + f(1-f) \, d\mu = 0
  \end{align*}
  where we used that $f_N, f \in \{0,1\}$ and $f_N = \mathbbold{1}_{\Phi_N}, f = \mathbbold{1}_{\Phi}$.
  As $\schvar_N \in [-1,1]$ one has 
  \begin{align*}
    \frac{4}{\sqrt{h}N} \sum_{x \in X_N} (\phi_N(x)  - \schvar_N(x))_+ &= \frac{4}{\sqrt{h}} \int_{\dom} (\phi_N(T_N(x))  - \schvar_N(T_N(x)))_+ \, d\mu(x)\\
   &= \frac{4}{\sqrt{h}} \int_{ \Phi_N} (1  - \schvar_N \circ T_N)_+ \, d\mu\\
     &= \frac{4}{\sqrt{h}} \int_{\Phi} (1  - \schvar_N \circ T_N)_+ \, d\mu +   \frac{4}{\sqrt{h}} \int_{\Phi_N\triangle \Phi} (1  - \schvar_N \circ T_N)_+ \, d\mu.
  \end{align*}
  This yields the convergence of the energy as 
  \begin{equation*}
    \frac{4}{\sqrt{h}} \int_{\Phi} (1  - \schvar_N \circ T_N)_+ \, d\mu  = \frac{4}{\sqrt{h}} \int_{\Phi} 1  - \schvar_N \circ T_N\, d\mu \xrightarrow{N \rightarrow \infty} \frac{4}{\sqrt{h}} \int_{\Phi} 1  - \tilde{\schvar}_\delta \, d\mu = \frac{4}{\sqrt{h}} \int_{\Omega} (\phi^\Phi  - \tilde{\schvar}_\delta)_+ \, d\mu 
  \end{equation*}
  and 
  \begin{equation*}
     \left|\frac{4}{\sqrt{h}} \int_{\Phi_N \triangle \Phi} (1  - \schvar_N \circ T_N)_+ \, d\mu\right| \leq \frac{8}{\sqrt{h}}|\Phi_N\triangle \Phi| \xrightarrow{N \rightarrow 0} 0. \qedhere
  \end{equation*}

\end{proof}

\section{Numerical Simulation of Invasion Processes}\label{sec:numerics}
\begin{figure}
  \includegraphics[width=0.19\linewidth]{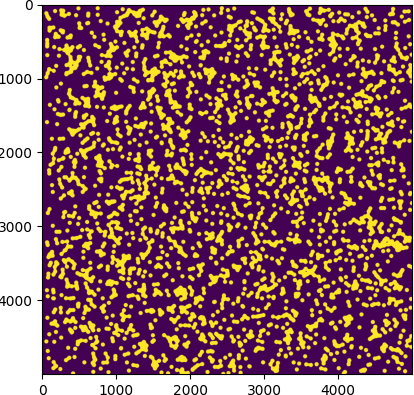}
  \includegraphics[width=0.19\linewidth]{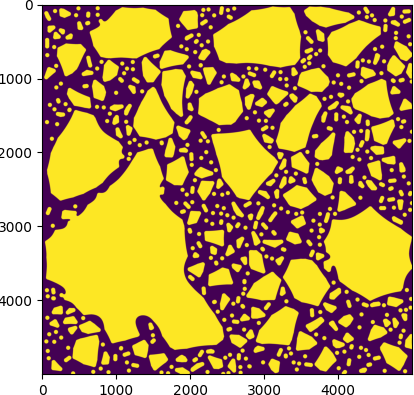}
  \includegraphics[width=0.19\linewidth]{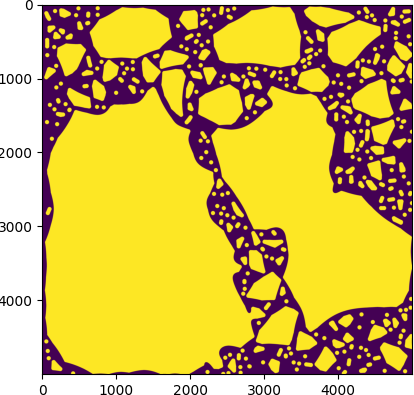}
  \includegraphics[width=0.19\linewidth]{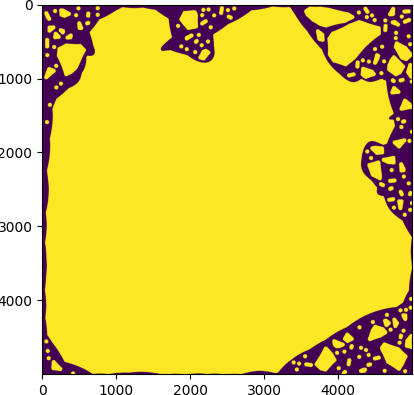}
  \includegraphics[width=0.19\linewidth]{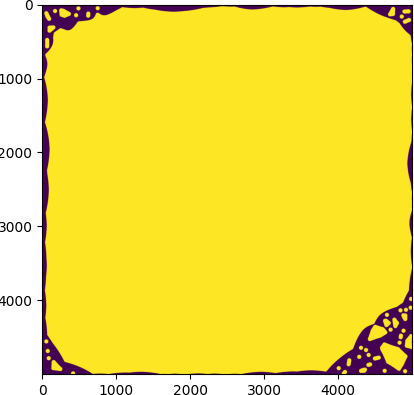}
  \caption{Simulation of mean curvature flow with obstacles computed with the scheme \eqref{def:discrete_scheme}. From left to right: initial condition, after 2000, 4000, 6000 and 8000 iterations. The initial conditions $u_0$ and the obstacles $\Phi$ are the same set ($u_0 = \mathbbold{1}_\Phi$) given by the union of uniformly random sampled discs.}
  \label{fig:snapshots}
\end{figure}
  We briefly describe our implementation of the dynamics of the invasion processes of~\cite{martin2023domain} in two dimensions. 
  As introduced in the previous section we work on the torus $\mathbb{T}^2 = [0,1)^2$ and discretize with a uniform grid as explained in Remark~\ref{rem:space_disc}ii). The two-dimensional Fourier transform of the heat kernel assosciated to~\eqref{def:finite_dif_laplace} can be explicitly computed to be
  \begin{equation*}
    \hat{G}_{ij} =\frac{1}{N} e^{2Nh(\cos(\frac{2\pi i}{\sqrt{N}}) + \cos(\frac{2\pi j}{\sqrt{N}}) - 2)}.
  \end{equation*}
  Here, $N = (1/\eps)^2$ is the number of grid points per row introduced by a uniform grid with distance $\eps$ between the grid points. The scaling $N$ is such that the width of the heat kernel is only dependent on $h$ but not on $\eps = 1/N$.
  Hence, solving the heat equation can be done efficiently by pointwise multiplication in Fourier space. Thus, updating $u^\ell \in \R^{N \times N}$ by the scheme \eqref{def:discrete_scheme} is simply given by
  \begin{equation}\label{eq:implemented_scheme}
  u^{\ell +1 } = \chi_{\Phi} \wedge \texttt{IFFT}(\hat{G} \odot \texttt{FFT}(u^\ell))\vee  (- \chi_{\Psi}),
\end{equation}
where \texttt{FFT} is the two-dimensional Fast Fourier Transform, \texttt{IFFT} is the two-dimensional Inverse Fast Fourier Transform and $\odot$ denotes pointwise multiplication. The overall running time is in $\mathcal{O}(N \log N)$ where the computational bottle neck is the Fast Fourier Transform. Fig.~\ref{fig:running_time} shows the average running time per iteration dependent on the number of pixels $N$. To measure the running time we used our Python implementation on a single AMD Ryzen 5 PRO 8540U. The code for all our tests can be found at~\cite{obstacleMBO}. The plot shows agreement with the theoretical (up to logarithmic correction) linear relation between $N$ and the complexity. 
  \begin{figure}
    \centering
    \includegraphics[height=5cm]{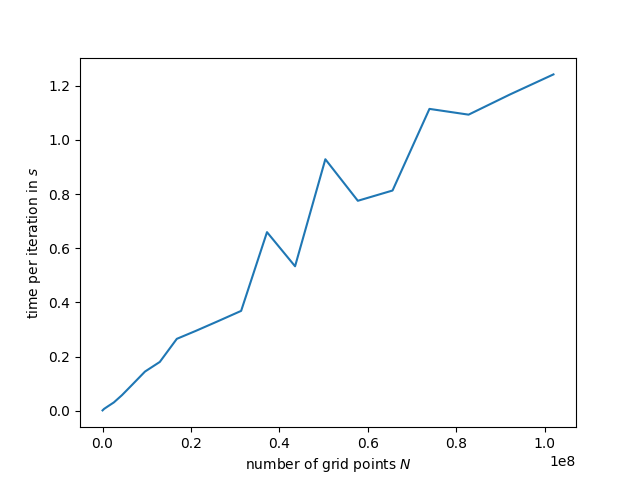}
    \caption{Comparison of average running time per iteration of our scheme described in Section~\ref{sec:numerics}. The average was taken over $100$ iterations.}
    \label{fig:running_time}
  \end{figure}
In Fig.~\ref{fig:convergence_study} we empirically study the convergence behavior of the scheme. We make two experiments: In the first, we apply the scheme to a single circle with radius $0.25$ and no obstacles as initial conditions. In this case, the solution to the mean curvature flow equation~\eqref{def:mcf} at time $t$ is given by a circle with radius $r(t) := \sqrt{-2t + 0.25^2}$. We compute an approximate radius $r_{N,h}(t) := \sqrt{\frac{\mu_N(\{ \schvar_{N,h}^{t/h} = 1\})}{\pi}}$ where $\mu_N$ is the empirical measure on the pixels of the $N\times N$ grid. Note that $t$ and $h$ are both powers of $2$ in our experiments which implies that $t/h \in \N$ is a valid iterant. We plot for $t = 1/2^6$ the approximation error $|r(t) - r_{N,h}(t)|$ for different values of $N$ and $h$ given by powers of $2$.

In the second experiment we choose an initial configuration with inner obstacles $\Phi$ (green) and outer obstacles $\Psi$ (blue), each given by two circles. No exact solution is known for such initial condition. Thus, we compute a fine approximation $\schvar_{\tilde{N},\tilde{h}}$ with $\tilde{N} = 2^{13}$ and $\tilde{h} = 1/2^{15}$ and take it as substitute to compare to instead of the exact solution. We plot the $L^1$ error $\int_{\T^2} |\schvar^{t/h}_{N,h} - \schvar^{t/h}_{\tilde{N},\tilde{h}}|\, d\mu_{\tilde{N}}$ for $t = 1/2^6$.

Our observations agree with~\cite{zbMATH06640199}: If $h \ll 1/N$ then the scheme is pinned meaning that it does not move over time. This happens because $h$ is too small for the interface to jump the distance from one grid point to another; that is, for given interface velocity $V$ it has to holds $Vh \gtrsim \eps = \frac{1}{N}$. In the regime $h \gg 1/N$ where this is satisfied, we observe an approximetely linear convergence in $h$. This  is in line with the fact that minimizing movement schemes convergence linearly to the corresponding gradient flow (provided the energy is convex). Furthermore, it seems like the optimal parameter regime (the valleys in our plots) correspond to picking $\sqrt{h} \approx \frac{1}{N}$.

\begin{figure}[t]
    \centering
    \includegraphics[width=0.28\linewidth]{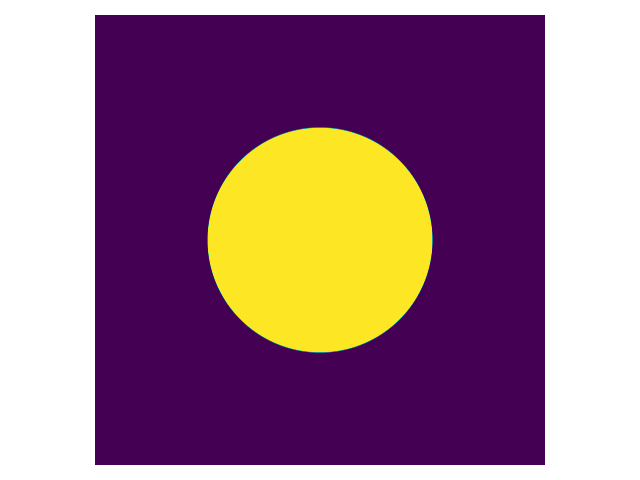}
    \includegraphics[width=0.28\linewidth]{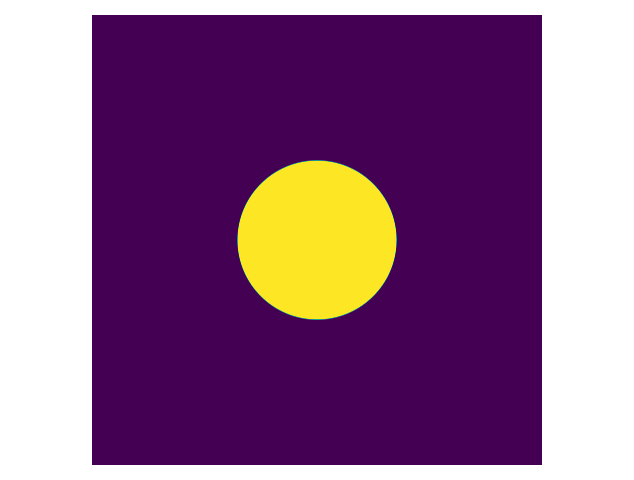}
    \includegraphics[width=0.4\linewidth]{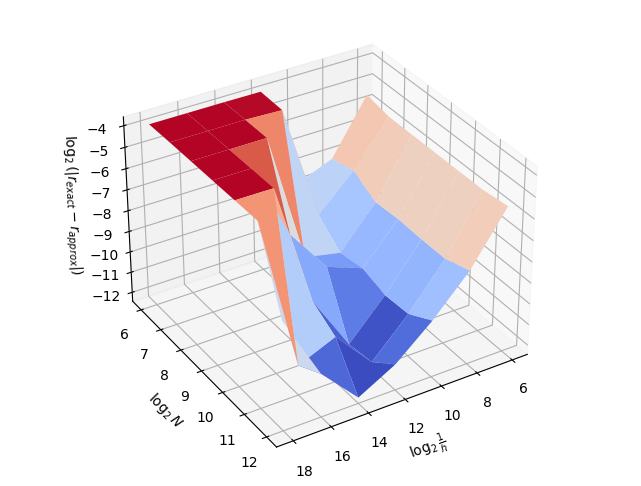}
    \vspace{0.5em}
    \includegraphics[width=0.28\linewidth]{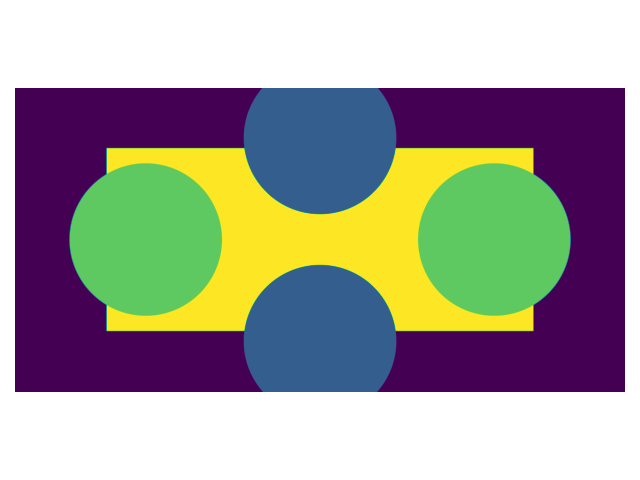}
    \includegraphics[width=0.28\linewidth]{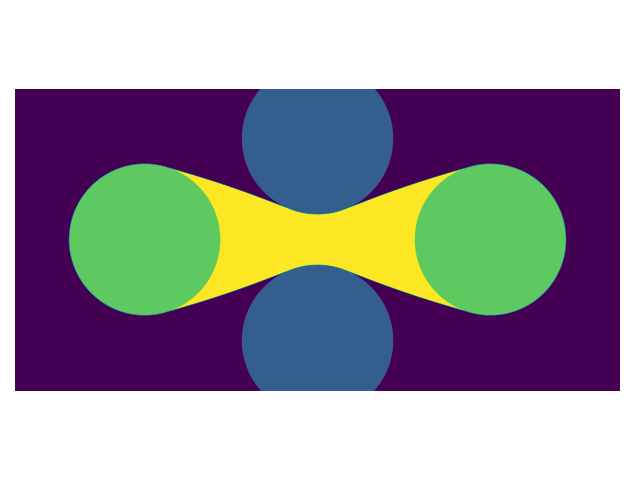}
    \includegraphics[width=0.4\linewidth]{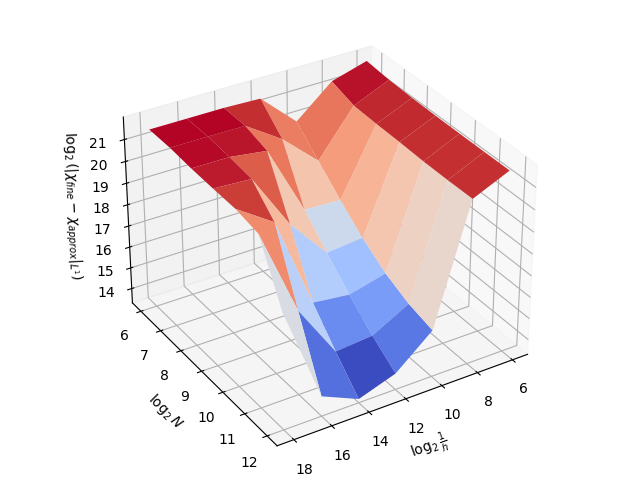}

    \caption{First row: Convergence study for a circle without obstacles. Left: Initial configuration. Middle: Configuration at time $t =  1/2^6$. Right: Surface plot of discretization error for different time discretizations $h$ and space discretizations $N$ (number of pixels per row). The displayed error is $|r_{exact} - r_{approx}|$, where $r_{exact}$ is the exact radius of the circle evolving by MCF, while $r_{approx}$ is the approximated radius given by the thresholding scheme. 
    Second row: Left: Initial condition with two inner (green) and two outer (blue) obstacles. Middle final configuration $\schvar_{fine}$ for good discretization values $N = 2^{13}$ and $h = 1/2^{15}$. Right: Plot of $L^1$ error comparing different discretizations $\schvar_{approx}$ to the more presize $\schvar_{fine}$.
    Observation: The plots show the pinning regime $h \ll 1/N$ where the scheme doesn't move and the error is maximal; the optimal convergence regime $\sqrt{h} \approx \frac{1}{N}$ visible as bottom of the valley; and the regime $h \gg 1/N$ where linear convergence in $h$ is expected.} 
  \label{fig:convergence_study}
\end{figure}

In view of approximating the dynamics of the invasion processes described by~\cite{martin2023domain}, two forms of approximation errors are visible in our experiments; compare with Fig.~\ref{fig:steady_states}. First, two components can get connected when they are close to each other although they might not touch yet. This happens due to the non-locality of the kernel. The length scale of this phenomenon is the width of the heat kernel $\sqrt{h}$. That means that the components get connected when they are roughly $\sqrt{h}$ apart. For the invasion process of~\cite{martin2023domain} this means that the front propagates faster than it should. On the other hand, a second effect slows down propagation; the dynamics can get pinned although the convex hull is not yet reached. For a steady state of obstacle mean curvature flow the interface should be straight away from the obstacle. But in our discrete dynamics with discrete normal velocity $V$ it has to hold $Vh \gtrsim \eps$ such that the interface can jump to the next grid point and not get pinned. As $V \approx -H$ this means that the discrete curvature $-H$ can be of the order $\frac{\eps}{h}$ for a discrete steady state. This observation is in line with~\cite{zbMATH06640199} where they prove convergence of the classical MBO scheme on grids when $\eps \ll h$.
  
  We want to simulate the dynamics of the invasion process of~\cite{martin2023domain}. This is done by sampling disks which serve as obstacles but also initial condition for the dynamics generated by~\eqref{eq:implemented_scheme}. To be more precise, we use a uniform grid, as in Remark~\ref{rem:space_disc}ii), on the two-dimensional torus $\mathbb{T}^2$ with $\eps=\frac{1}{5000}$ wich leads to $N = 5000^2$. We choose the disk radius $r_d = 0.005642$ such that the dimensionless system size is $\bar{A}_{syst} :=\frac{|\T^2|}{\pi r_d^2} = 10000$. Furthermore, we choose the number of disks for the initialization to be $N_d = 2999$ such that the initial concentration is $N_d \pi r_d^2=:C = 0.3$. The centers of the disks are uniformly randomly distributed over the domain. The characteristic function of the disks $\chi_d(x) := \begin{cases}
  1 & \text{if } x \text{ is in a disk},\\
  -1 & \text{otherwise},
  \end{cases}$ is used as initial condition $\chi^0$ and also as inner obstacle $\chi_\Phi$. Note, we also add a zero padding to avoid changes in the dynamics due to the periodic boundary conditions. The size of the padding depends on the diffusion time. The right choice of diffusion time depends on $N$ and the desired approximation error as previously discussed. In our tests $h = r_{\text{disk}}^2/16$ seemed to be a fair trade between approximation accuracy and number of iterations needed. Note, the choice for $h$ is such that the width of the kernel is a quarter of the width of the obstacles. Some snapshots of the simulation can be seen in Fig.~\ref{fig:snapshots}.

  \begin{figure}[h]
\centering

\subfigure[Initialization]{  \label{subfig:a}
  \begin{adjustbox}{width=0.19\linewidth}
    
\begin{tikzpicture}[x=0.75pt,y=0.75pt,yscale=-1,xscale=1]

\draw  [fill={rgb, 255:red, 68; green, 1; blue, 84 }  ,fill opacity=1 ] (30.33,4.8) -- (295.51,4.8) -- (295.51,272.28) -- (30.33,272.28) -- cycle ;
\draw  [color={rgb, 255:red, 253; green, 231; blue, 36 }  ,draw opacity=1 ][fill={rgb, 255:red, 253; green, 231; blue, 36 }  ,fill opacity=1 ] (71.43,93.44) .. controls (71.43,68.95) and (91.26,49.09) .. (115.71,49.09) .. controls (140.17,49.09) and (160,68.95) .. (160,93.44) .. controls (160,117.93) and (140.17,137.78) .. (115.71,137.78) .. controls (91.26,137.78) and (71.43,117.93) .. (71.43,93.44) -- cycle ;
\draw [color={rgb, 255:red, 255; green, 255; blue, 255 }  ,draw opacity=1 ] [dash pattern={on 7.5pt off 2.25pt}]  (160.33,2.67) -- (160.33,273.67) ;
\draw [color={rgb, 255:red, 255; green, 255; blue, 255 }  ,draw opacity=1 ]   (160.33,242.67) -- (181.33,243) ;
\draw    (30,272.02) -- (24,272.02) ;
\draw    (30,4.98) -- (24,4.98) ;
\draw    (30,279.53) -- (30,272) ;
\draw    (295,280) -- (295,272.47) ;
\draw  [color={rgb, 255:red, 253; green, 231; blue, 36 }  ,draw opacity=1 ][fill={rgb, 255:red, 253; green, 231; blue, 36 }  ,fill opacity=1 ] (71.43,182.12) .. controls (71.43,157.63) and (91.26,137.78) .. (115.71,137.78) .. controls (140.17,137.78) and (160,157.63) .. (160,182.12) .. controls (160,206.61) and (140.17,226.47) .. (115.71,226.47) .. controls (91.26,226.47) and (71.43,206.61) .. (71.43,182.12) -- cycle ;
\draw  [color={rgb, 255:red, 253; green, 231; blue, 36 }  ,draw opacity=1 ][fill={rgb, 255:red, 253; green, 231; blue, 36 }  ,fill opacity=1 ] (181.33,138.5) .. controls (181.33,114.01) and (201.16,94.16) .. (225.62,94.16) .. controls (250.08,94.16) and (269.9,114.01) .. (269.9,138.5) .. controls (269.9,162.99) and (250.08,182.84) .. (225.62,182.84) .. controls (201.16,182.84) and (181.33,162.99) .. (181.33,138.5) -- cycle ;
\draw  [fill={rgb, 255:red, 0; green, 0; blue, 0 }  ,fill opacity=1 ] (222.29,138.5) .. controls (222.29,136.66) and (223.78,135.17) .. (225.62,135.17) .. controls (227.46,135.17) and (228.95,136.66) .. (228.95,138.5) .. controls (228.95,140.34) and (227.46,141.83) .. (225.62,141.83) .. controls (223.78,141.83) and (222.29,140.34) .. (222.29,138.5) -- cycle ;
\draw  [dash pattern={on 7.5pt off 2.25pt}]  (225.62,138.5) -- (208.2,138.5) -- (181.33,138.5) ;
\draw [color={rgb, 255:red, 255; green, 255; blue, 255 }  ,draw opacity=1 ] [dash pattern={on 7.5pt off 2.25pt}]  (181.33,3) -- (181.33,274) ;

\draw (159.9,225.71) node [anchor=north west][inner sep=0.75pt]    {$\textcolor[rgb]{1,1,1}{0.1}$};
\draw (25,282.4) node [anchor=north west][inner sep=0.75pt]    {$0$};
\draw (6,-3.88) node [anchor=north west][inner sep=0.75pt]    {$1$};
\draw (8,263.4) node [anchor=north west][inner sep=0.75pt]    {$0$};
\draw (290,282.4) node [anchor=north west][inner sep=0.75pt]    {$1$};
\draw (186.9,114.71) node [anchor=north west][inner sep=0.75pt]    {$0.1\overline{6}$};

\end{tikzpicture}
  \end{adjustbox}}%
  \subfigure[Ground truth]{  \label{subfig:ab}
  \begin{adjustbox}{width=0.192\linewidth}
    
\begin{tikzpicture}[x=0.75pt,y=0.75pt,yscale=-1,xscale=1]

\draw  [fill={rgb, 255:red, 68; green, 1; blue, 84 }  ,fill opacity=1 ] (30.33,4.8) -- (295.51,4.8) -- (295.51,272.28) -- (30.33,272.28) -- cycle ;
\draw  [color={rgb, 255:red, 253; green, 231; blue, 36 }  ,draw opacity=1 ][fill={rgb, 255:red, 248; green, 231; blue, 28 }  ,fill opacity=1 ] (71.43,93.44) .. controls (71.43,68.95) and (91.26,49.09) .. (115.71,49.09) .. controls (140.17,49.09) and (160,68.95) .. (160,93.44) .. controls (160,117.93) and (140.17,137.78) .. (115.71,137.78) .. controls (91.26,137.78) and (71.43,117.93) .. (71.43,93.44) -- cycle ;
\draw    (30,272.02) -- (24,272.02) ;
\draw    (30,4.98) -- (24,4.98) ;
\draw    (30,279.53) -- (30,272) ;
\draw    (295,280) -- (295,272.47) ;
\draw  [color={rgb, 255:red, 253; green, 231; blue, 36 }  ,draw opacity=1 ][fill={rgb, 255:red, 253; green, 231; blue, 36 }  ,fill opacity=1 ] (71.43,182.12) .. controls (71.43,157.63) and (91.26,137.78) .. (115.71,137.78) .. controls (140.17,137.78) and (160,157.63) .. (160,182.12) .. controls (160,206.61) and (140.17,226.47) .. (115.71,226.47) .. controls (91.26,226.47) and (71.43,206.61) .. (71.43,182.12) -- cycle ;
\draw  [color={rgb, 255:red, 253; green, 231; blue, 36 }  ,draw opacity=1 ][fill={rgb, 255:red, 253; green, 231; blue, 36 }  ,fill opacity=1 ] (181.33,138.5) .. controls (181.33,114.01) and (201.16,94.16) .. (225.62,94.16) .. controls (250.08,94.16) and (269.9,114.01) .. (269.9,138.5) .. controls (269.9,162.99) and (250.08,182.84) .. (225.62,182.84) .. controls (201.16,182.84) and (181.33,162.99) .. (181.33,138.5) -- cycle ;
\draw  [color={rgb, 255:red, 248; green, 231; blue, 28 }  ,draw opacity=1 ][fill={rgb, 255:red, 248; green, 231; blue, 28 }  ,fill opacity=1 ] (71.43,93.44) -- (160,93.44) -- (160,182.12) -- (71.43,182.12) -- cycle ;

\draw (23,279) node [anchor=north west][inner sep=0.75pt]    {$0$};
\draw (4,-7.28) node [anchor=north west][inner sep=0.75pt]    {$1$};
\draw (6,260) node [anchor=north west][inner sep=0.75pt]    {$0$};
\draw (288,279) node [anchor=north west][inner sep=0.75pt]    {$1$};

\end{tikzpicture}
  \end{adjustbox}}%
\subfigure[$h=0.00001$]{\label{subfig:b}\includegraphics[width=0.195\linewidth]{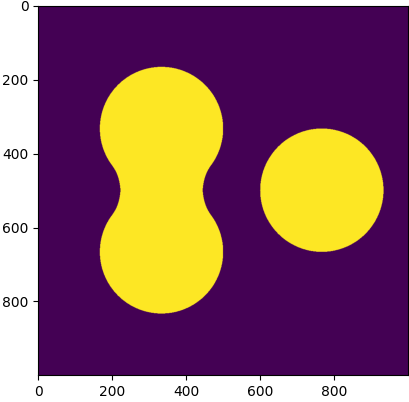}}%
\subfigure[$h=0.00085$]{\label{subfig:c}\includegraphics[width=0.195\linewidth]{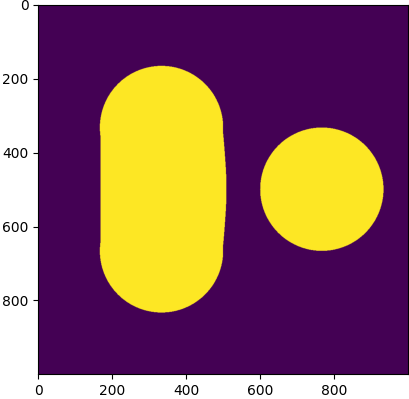}}%
\subfigure[$h=0.0009$]{\label{subfig:d} \includegraphics[width=0.195\linewidth]{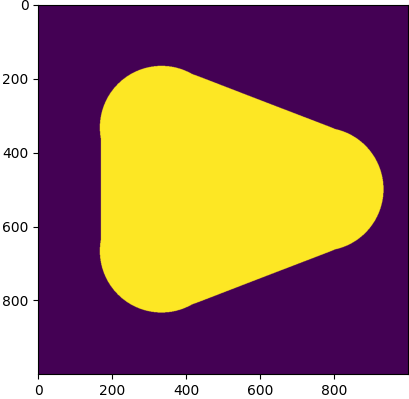}}%
\caption{Simualation to visualize numerical errors in the steady states which are caused by the time and space discretization. Fig.~\ref{subfig:a} shows schematically the initial configuration $u_0$, which is at the same time the obstacle $\Phi$, i.e., $u_0 = \mathbbold{1}_\Phi$. Fig.~\ref{subfig:ab} shows the steady state of the mean curvature flow with obstacles with initial condition Fig.~\ref{subfig:a}. Fig.~\ref{subfig:b}, \ref{subfig:c} and \ref{subfig:d} show the steady states of our scheme~\eqref{eq:implemented_scheme} for different diffusion times $h$ but fixed spatial discretization $\eps = 0.001$. In  Fig.~\ref{subfig:b} the diffusion time $h$ is not big enough for the given spatial resolution such that the convex hull - the steady state of the obstacle mean curvature flow - is not well approximated. In  Fig.~\ref{subfig:c} $h$ is chosen well for the example. In  Fig.~\ref{subfig:d} $h$ is too big such that non-touching areas get connected.} 
\label{fig:steady_states}
\end{figure}

\section{Outlook}
 Our convergence proofs show that combining the MBO scheme with a pointwise update by the obstacle constraint yields an efficient algorithm for the approximation of mean curvature flow with obstacles. We believe that also other variants of the MBO scheme can be similarly adapted to incoperate an obstacle constraint. This is relevant for applicatations. For example the model of~\cite{elliott2012modelling} incorporates, among other effects, obstacles but also volume constraints on the cell. This can be approximated by adapting the volume-constrained MBO scheme~\cite{zbMATH06841827, arXiv:2412.17694, arXiv:2508.09064}. The scheme for two phases then reads:
 \begin{itemize}
  \item[\textbf{Step 1}] Diffusion: $\schvar^{\ell + 1/3} = e^{h \Delta}$
  \item[\textbf{Step 2}] Update on obstacle: $\schvar^{\ell + 2/3} = \chi_{\Phi} \wedge \schvar^{\ell + 1/3} \vee (- \chi_{\Psi})$
  \item[\textbf{Step 3}] Threshold with volume constraint: $\schvar^{\ell + 1} = \begin{cases}
    1 & \text{if } \schvar^{\ell + 2/3} > \lambda,\\
    -1 & \text{if } \schvar^{\ell + 2/3} < \lambda,
  \end{cases}$ where $\lambda$ is a thresholding value such that $|\{\schvar^{\ell + 2/3} > \lambda\}| = V$ for a given volume $V$.
\end{itemize}
The Lagrange multiplier $\lambda$ can be found in $\mathcal{O}(N)$ with standard weighted median algorithms~\cite[Theorem 17.3]{zbMATH05993956}.
 The above scheme is so powerful because it keeps the minimizing movement interpretation
  \begin{align}
    \schvar^{\ell + 1} &\in \argmin_{u:\dom \rightarrow \{-1,1\}} E_h(u) + \frac{1}{2h}d_h^2(\schvar, \schvar^\ell) + \frac{2\lambda}{\sqrt{h}} \int \schvar \quad \quad \text{ s.t. } \quad \quad \chi_\Phi \leq \schvar \leq \chi_\Psi , \label{eq:volume_lagrange}
  \end{align}
  or equivalently 
  \begin{align*}
    \schvar^{\ell + 1} &\in \argmin_{u:\dom \rightarrow \{-1,1\}} E_h(u) + \frac{1}{2h}d_h^2(\schvar, \schvar^\ell)  \quad \quad \text{ s.t. } \quad \quad \chi_\Phi \leq \schvar \leq \chi_\Psi \quad \text{ and } \quad |\{u = 1\}| = V.
  \end{align*}
 This again implies the unconditional stability and the relation to the gradient flow structure of mean curvature flow. Although, adapting the algorithms seems to be straight forward, the analytical side for volume-preserving mean curvature flow with obstacles is not studied yet; so, there is still more work to do to understand the volume preserving, obstacle MBO.

Furthermore, a multiphase version of the our obstacle MBO scheme would be a natural extension. The obstacles $\Phi_i$ are then inner obstacles for the phase $i$ and outer obstacles for any other phase $j\neq i$. Again, this can be done with volume constraints, such that the minimizing movement interpretation~\eqref{eq:volume_lagrange} is kept. The Lagrange multiplier $\lambda$ and consequently a minimizer $u^{\ell + 1}$ can be computed by using one of the algorithms~\cite{zbMATH06841827, arXiv:2412.17694}. 

The concept of viscosity solutions is not available for multiphase, obstacle mean curvature flow, as this flow does not satisfy a comparison principle. Instead one should tackle the multiphase problem by its gradient flow structure and show convergence to Brakke~\cite{zbMATH07161205}, De Giorgi~\cite{zbMATH07373006, arXiv:2412.17694} or distributional~\cite{zbMATH06654682} solutions. The main challenge to that aim is to understand the convergence of the Lagrange multiplier in the minimizing movement interpretation~\eqref{eq:minmovlag_cont}.

  A second meaningful adaptation would be anisotropic motions. Mean curvature flow originates from materials science where it models the motion of grain boundaries in polycrystals. Here, obstacles are important to model impurities in the material. Moreover, due to the microscopic crystal structure the motion is actually anisotropic. Therefore, it is of interest to adapt the MBO scheme for anisotropic mean curvature flow with obstacles. The natural approach to adapt the scheme \eqref{def:scheme} is to replace the heat kernel with an anisotropic kernel. For the anisotropic MBO scheme without obstacles there are already several papers answering the question on how to choose the kernel. First, \cite{zbMATH01336919} consider a broad class of positive kernels. But their kernels may not be positive in Fourier space such that there is no guarantee of energy dissipation in~\eqref{eq:minmov_cont}. \cite{zbMATH06049950} study kernels with positive Fourier transform but the kernels may not be positive in physical space such that the monotonicity of the scheme may break down. Furthermore, \cite{zbMATH06864183} characterize which anisotropies can be achieved as limits of the MBO scheme with positive kernels; they also showed that there exist in physical and in Fourier space positive kernels for those anisotropies. Nevertheless, the analysis for those kernels with an obstacle constraint is still an exciting, completely open problem. 

\section*{Acknowledgments}
This project is funded by the Deutsche Forschungsgemeinschaft (DFG, German Research Foundation) under Germany's Excellence Strategy EXC 2181/1 - 390900948 (the Heidelberg STRUCTURES Excellence Cluster). The authors would like to thank Keisuke Takasao for many fruitful discussions and for his insightful comments on topics related to obstacle mean curvature flow.

\frenchspacing
\bibliographystyle{abbrv}
\bibliography{references}
  \end{document}